\documentclass[12pt]{amsart}

\usepackage{amsmath,amssymb}
\usepackage{bm}
\usepackage{graphicx}
\usepackage{ascmac}
\usepackage{amsthm}
\usepackage{amsfonts}
\usepackage{mathrsfs}
\usepackage{mathtools}
\usepackage{comment}
\usepackage[all]{xy}

\theoremstyle{plain}
\newtheorem{thm}{Theorem}[section]
\newtheorem{cor}[thm]{Corollary}
\newtheorem{prop}[thm]{Proposition}
\newtheorem{lem}[thm]{Lemma}

\theoremstyle{definition}
\newtheorem{dfn}[thm]{Definition}
\newtheorem{ex}[thm]{Example}
\newtheorem{rem}[thm]{Remark}
\newtheorem{qst}[thm]{Question}
\newtheorem*{nota}{Notation}

\newtheorem*{oftp}{Organization of this paper}
\newtheorem*{ack}{Acknowledgments}

\numberwithin{thm}{section}
\numberwithin{equation}{section}

\newcommand{\Zpn}{\mathbb{Z}_{>0}}
\newcommand{\Znn}{\mathbb{Z}_{\geq 0}}
\newcommand{\Z}{\mathbb{Z}}

\newcommand{\Q}{\mathbb{Q}}

\newcommand{\Fp}{\mathbb{F}_p}

\newcommand{\A}{\mathbb{A}}

\newcommand{\F}{\mathbb{F}}
\newcommand{\G}{\mathbb{G}}

\newcommand{\sS}{\mathscr{S}}

\renewcommand{\hbar}{\overline{h}}

\newcommand{\Fpt}{\mathbb{F}^{\times}_p}

\newcommand{\cC}{\mathcal{C}}
\newcommand{\cD}{\mathcal{D}}

\newcommand{\cH}{\mathcal{H}}
\newcommand{\cI}{\mathcal{I}}

\newcommand{\fA}{\mathfrak{A}}
\newcommand{\fS}{\mathfrak{S}}

\DeclareMathOperator{\Ker}{Ker}
\DeclareMathOperator{\Ima}{Im}
\DeclareMathOperator{\Coker}{Coker}
\DeclareMathOperator{\Hom}{Hom}

\DeclareMathOperator{\id}{id}
\DeclareMathOperator{\Aut}{Aut}

\DeclareMathOperator{\ord}{ord}
\DeclareMathOperator{\Gal}{Gal}

\DeclareMathOperator{\GL}{GL}

\DeclareMathOperator{\SL}{SL}

\DeclareMathOperator{\ab}{ab}
\DeclareMathOperator{\N}{N}

\DeclareMathOperator{\loc}{loc}

\DeclareMathOperator{\Res}{Res}

\DeclareMathOperator{\Inf}{Inf}

\DeclareMathOperator{\Br}{Br}

\DeclareMathOperator{\Pic}{Pic}

\DeclareMathOperator{\ad}{ad}

\DeclareMathOperator{\sep}{sep}
\DeclareMathOperator{\der}{der}

\DeclareMathOperator{\Ind}{Ind}
\DeclareMathOperator{\pr}{pr}

\DeclareMathOperator{\Ad}{Ad}

\DeclareMathOperator{\et}{\acute{e}t}

\DeclareMathOperator{\Sel}{Sel}

\DeclareMathOperator{\Inn}{Inn}
\DeclareMathOperator{\Out}{Out}

\usepackage[OT2,T1]{fontenc}
\DeclareSymbolFont{cyrletters}{OT2}{wncyr}{m}{n}
\DeclareMathSymbol{\Sha}{\mathalpha}{cyrletters}{"58}

\usepackage{geometry}
\title[Hasse norm principle]{Hasse norm principle for extensions of prime squared degree}
\author[Y.~Oki]{Yasuhiro Oki}
\address{Department of Mathematics, College of Science, Rikkyo University, 3-34-1, Nishi-Ikebukuro, Toshima-ku, Tokyo, 171-8501, Japan. }
\email{oki@rikkyo.ac.jp}

\begin{document}

\begin{abstract}
We give an equivalent condition for the validity of the Hasse norm principle for finite separable extensions of prime squared degree of global fields. Our theorem recovers the result of Drakokhrust--Platonov, which claims that the Hasse norm principle holds for adequate extensions of prime squared degree. 
\end{abstract}

\maketitle

\tableofcontents

\section{Introduction}\label{intr}

Let $k$ be a global field. For a finite separable field extension $K/k$, put
\begin{equation*}
\Sha(K/k):=(k^{\times}\cap \N_{K/k}(\A_{K}))/\N_{K/k}(K^{\times}),
\end{equation*}
where $\A_{K}^{\times}$ is the id{\`e}le group of $K$. In the paper of Scholz (\cite{Scholz1940}), the group $\Sha(K/k)$ is denoted by $\kappa(K/k)$ and is called the \emph{number knot}. We say that the \emph{Hasse norm principle holds for $K/k$} if
\begin{equation*}
\Sha(K/k)=1. 
\end{equation*}
The study of $\Sha(K/k)$ is initiated by Hasse (\cite{Hasse1931}). More precisely, he proved that the Hasse norm principle holds for all cyclic extensions. However, $\Sha(K/k)$ is not always trivial in general. Indeed, he also clarified that the biquadratic extension $\Q(\sqrt{-39},\sqrt{-3})/\Q$ does not satisfy the Hasse norm principle. Hence, it is natural to discuss an equivalent condition for the validity of the Hasse norm principle. In this paper, we consider the following: 

\begin{qst}\label{qst:hnfd}
For $d\in \Zpn$, determine an equivalent condition on $K/k$ of degree $d$ for which the Hasse norm principle holds. 
\end{qst}
This has been fully resolved in the cases that
\begin{itemize}
\item $d$ is a prime number (\cite{Bartels1981a}); 
\item $d=4$ (\cite{Kunyavskii1984}); 
\item $d=6$ (\cite{Drakokhrust1987}); and
\item $8\leq d\leq 15$ (\cite{Hoshi2022}, \cite{Hoshi2023}). 
\end{itemize}
In particular, the result for $d=4$ can be stated as follows. 

\begin{thm}[{cf.~\cite[p.~1899]{Kunyavskii1984}}]\label{thm:knvs}
Let $k$ be a global field. Consider a finite separable field extension $K/k$ of degree $4$ with Galois closure $\widetilde{K}/k$. Put $G:=\Gal(\widetilde{K}/k)$ and $H:=\Gal(\widetilde{K}/K)$. 
\begin{enumerate}
\item If $\Sha(K/k)\neq 1$, then we have 
\begin{itemize}
\item[(a)] $G\cong V_{4}$, the Klein four group; or 
\item[(b)] $G\cong \fA_{4}$, the alternating group of $4$-letters. 
\end{itemize}
\item If \emph{(a)} or \emph{(b)} is valid, then there is an isomorphism
\begin{equation*}
\Sha(K/k)\cong 
\begin{cases}
1&\text{if a decomposition group of $\widetilde{K}/k$ contains $(C_{2})^{2}$; }\\
\Z/2 &\text{otherwise. }
\end{cases}
\end{equation*}
\end{enumerate}
\end{thm}

In other cases, there exist numerous partial results on the above problem, including \cite{Tate1967}, \cite{Gurak1978}, \cite{Bartels1981b}, \cite{Voskresenskii1984}, \cite{Drakokhrust1987}, \cite{Macedo2020}, \cite{Hoshi2025}, and \cite{Oki2025a}. However, it is far from a complete solution.

This paper resolves Question \ref{qst:hnfd} in the case where $d$ is an odd prime squared, that is, $d=p^{2}$ for some odd prime number $p$. We denote by $\GL_{2}(\Fp)$ the group of matrices of size $2$ with coefficients $\Fp$, and set
\begin{equation*}
\SL_{2}(\Fp):=\{g\in \GL_{2}(\Fp)\mid \det(g)=1\}. 
\end{equation*}
On the other hand, we use the notations as follows for any positive integer $n$: 
\begin{itemize}
\item $C_{n}$ is the cyclic group of order $n$; and
\item $\fS_{n}$ is the symmetric group of $n$-letters; and
\item $\fA_{n}$ is the alternating group of $n$-letters. 
\end{itemize}
On the other hand, the \emph{standard representation of $\GL_{2}(\Fp)$} (resp.~$\SL_{2}(\Fp)$) refers to the $2$-dimensional representation of $\SL_{2}(\Fp)$ over $\Fp$ associated to the natural inclusion into $\GL_{2}(\Fp)$ (resp.~$\SL_{2}(\Fp)$). 

\begin{thm}[Theorem \ref{thm:hnps}]\label{mth1}
Let $p>2$ be an odd prime number, and $k$ a global field. Consider a finite separable field extension $K/k$ of degree $p^{2}$ with Galois closure $\widetilde{K}/k$. Put $G:=\Gal(\widetilde{K}/k)$ and $H:=\Gal(\widetilde{K}/K)$. 
\begin{enumerate}
\item If $\Sha(K/k)\neq 1$, then the following is valid: 
\begin{itemize}
\item[($\ast$)] there exists a subgroup $G^{\dagger}$ of $\SL_{2}(\Fp)$ such that 
\begin{equation*}
G\cong (C_{p})^{2}\rtimes_{\varphi} G^{\dagger},\quad 
H\cong \{1\}\rtimes_{\varphi} G^{\dagger}, 
\end{equation*}
where $\varphi$ is induced by the standard representation of $\SL_{2}(\Fp)$. 
\end{itemize}
\item If \emph{($\ast$)} holds, then there is an isomorphism
\begin{equation*}
\Sha(K/k)\cong 
\begin{cases}
1&\text{if a decomposition group of $\widetilde{K}/k$ contains $(C_{p})^{2}$; }\\
\Z/p &\text{otherwise. }
\end{cases}
\end{equation*}
\end{enumerate}
\end{thm}

Combining Theorem \ref{mth1} with Theorem \ref{thm:knvs}, we obtain a complete solution to Question \ref{qst:hnfd} for arbitrary prime squared $d$. 

\begin{rem}
Subgroups of $\SL_{2}(\Fp)$ are completely classified. See \cite[Chapter 3, \S 6]{Suzuki1982} for detail. More precisely, if $G^{\dagger}$ is a proper subgroup of $\SL_{2}(\Fp)$, then one of the following is satisfied: 
\begin{itemize}
\item $G^{\dagger}\cong C_{p}\rtimes_{\varphi} C_{d}$, where $d$ is a prime divisor of $p-1$, and $\varphi$ is induced by the composite
\begin{equation*}
C_{d}\xrightarrow{c\mapsto c^{2}}C_{d}\hookrightarrow \Aut(C_{p}); 
\end{equation*}
\item $G^{\dagger}\cong C_{d}$, where $d\in \Zpn$ is a divisor of $p-1$ or $p+1$; 
\item $G^{\dagger}\cong C_{d}\rtimes_{\varphi} C_{4}$, where $d\in \Zpn$ is an \emph{odd} divisor of $p-1$ or $p+1$, and $\varphi$ is induced by the composite
\begin{equation*}
C_{4}\twoheadrightarrow C_{2}\hookrightarrow \Aut(C_{d}); 
\end{equation*}
\item $G^{\dagger}\cong \langle g,h \mid g^{d}=1,h^{2}=g^{n/2} \rangle$, where $d\in \Zpn$ is an \emph{even} divisor of $p-1$ or $p+1$; 
\item $p>3$ and $G^{\dagger}\cong \SL_{2}(\F_{3})$ (a Schur cover of $\fA_{4}$); 
\item $p>3$ and $G^{\dagger}\cong \widehat{\Sigma}_{4}$, where $\widehat{\Sigma}_{4}$ is the Schur cover of $\fS_{4}$ as in \cite[Chapter 3, \S 6, Theorem 6.17 (iv)]{Suzuki1982}; 
\item $p\equiv \pm 1\bmod 5$ and $G^{\dagger}\cong \SL_{2}(\F_{5})$, the unique Schur cover of $\fA_{5}$. 
\end{itemize}
Note that our proof of Theorem \ref{mth1} does not use such a classification. 
\end{rem}

\begin{ex}
Here, we consider Theorem \ref{mth1} for $p=3$. Then, a subgroup of $\SL_{2}(\F_{3})$ is conjugate to one of the following: 
\begin{equation}\label{eq:sgtt}
\begin{gathered}
\left\{\begin{pmatrix}1&0\\0&1\end{pmatrix}\right\},\quad
\bigg\langle \begin{pmatrix}-1&0\\0&-1\end{pmatrix}\bigg\rangle,\quad
\bigg\langle \begin{pmatrix}1&1\\0&1\end{pmatrix}\bigg\rangle,\quad
\bigg\langle \begin{pmatrix}0&-1\\1&0\end{pmatrix}\bigg\rangle,\\
\bigg\langle \begin{pmatrix}-1&-1\\0&-1\end{pmatrix}\bigg\rangle,\quad
\bigg\langle \begin{pmatrix}0&-1\\1&0\end{pmatrix},
\begin{pmatrix}-1&-1\\-1&1\end{pmatrix}\bigg\rangle,\quad
\SL_{2}(\F_{3}). 
\end{gathered}
\end{equation}
These have orders $1$, $2$, $3$, $4$, $6$ and $8$, respectively. Moreover, the semi-direct products $(C_{p})^{2}\rtimes_{\varphi}G^{\dagger}$, where $G^{\dagger}$ are as in \eqref{eq:sgtt}, define the transitive groups $9T2$, $9T5$, $9T7$, $9T9$, $9T11$, $9T24$, and $9T23$. Hence, Theorem \ref{mth1} for $p=3$ recovers \cite[Theorem 1.16]{Hoshi2022} for $n=9$. Note that we give an alternative proof of the above previous result that does not use computer. 
\end{ex}

\begin{rem}
Theorem \ref{mth1} recovers the result of Drakokhrust and Platonov (\cite[Theorem 4]{Drakokhrust1987}), which asserts that the Hasse norm principle for $K/k$ of degree $p^{2}$ with $p$ an odd prime if $K$ is $k$-adequate in the sense of \cite[p.~451, l.~11--13, Definition]{Schacher1968}. See Theorem \ref{thm:dpes}. 
\end{rem}

For a finite separable field extension $K/k$, let
\begin{equation*}
T_{K/k}:=\Ker(\N_{K/k}\colon \Res_{K/k}\G_{m,K}\rightarrow \G_{m,k}). 
\end{equation*} 
It is called the \emph{norm one torus} attached to $K/k$. Note that it is also denoted by $R_{K/k}^{(1)}\G_{m}$. Take a smooth compactification $X$ of $T_{K/k}$ over $k$, and put $\overline{X}:=X\otimes_{k}k^{\sep}$. Then, it is known that $H^{1}(k,\Pic(\overline{X}))$ is isomorphic to a birational invariant called the \emph{unramified Brauer group} of the function field of $X$. 

\begin{thm}[{Theorem \ref{thm:hopc}}]\label{mth2}
Let $p>2$ be an odd prime number. Consider a finite separable field extension $K/k$ of degree $p^{2}$ with its Galois closure $\widetilde{K}/k$. Put $G:=\Gal(\widetilde{K}/k)$ and $H:=\Gal(\widetilde{K}/k)$. Take a smooth compactification $X$ of $T_{K/k}$, and put $\overline{X}:=X\otimes_{k}k^{\sep}$. Then the following are equivalent: 
\begin{enumerate}
\item $H^{1}(k,\Pic(\overline{X}))\neq 0$; 
\item $H^{1}(k,\Pic(\overline{X}))\cong \Z/p$; 
\item there exists a subgroup $G^{\dagger}$ of $\SL_{2}(\Fp)$ such that
\begin{equation*}
G\cong (C_{p})^{2}\rtimes_{\varphi} G^{\dagger},\quad H\cong \{1\}\rtimes_{\varphi} G^{\dagger}, 
\end{equation*}
where $\varphi$ is induced by the standard representation of $\SL_{2}(\Fp)$. 
\end{enumerate}
\end{thm}

The cohomology group $H^{1}(k,\Pic(\overline{X}))$ is also related to the Hasse norm principle. Here, assume that $k$ is a global field, and consider a finite separable extension $K/k$ . Denote by $\Sigma_{k}$ the set of all places of $k$. Then, Ono (\cite[p.~70]{Ono1963}) constructed an isomorphism
\begin{equation*}
\Sha(K/k)\cong \Sha^{1}(k,T_{K/k}). 
\end{equation*}
Here, the right-hand side is the \emph{Tate--Shafarevich group} of $T_{K/k}$, which is defined as the kernel of the homomorphism
\begin{equation*}
(\Res_{k_{v}/k})_{v\in \Sigma_{k}}\colon H^{1}(k,T_{K/k})\rightarrow \prod_{v\in \Sigma}H^{1}(k_{v},T_{K/k}). 
\end{equation*}
On the other hand, let $X$ and $\overline{X}$ be as before. Then, there is an exact sequence of finite groups
\begin{equation*}
0\rightarrow A_{k}(T_{K/k})\rightarrow H^{1}(k,\Pic(\overline{X}))^{\vee}\rightarrow \Sha^{1}(k,T_{K/k})\rightarrow 0. 
\end{equation*}
Here, $A_{k}(T_{K/k})$ is the defect of the weak approximation of $T_{K/k}$, that is, 
\begin{equation*}
A_{k}(T_{K/k}):=\left(\prod_{v\in \Sigma_{k}}T_{K/k}(k_{v})\right)/\overline{T_{K/k}(k)}. 
\end{equation*}

Therefore, combining Theorem \ref{mth1} with \ref{mth2}, we also obtain a structure theorem on $A_{k}(T_{K/k})$ in the case $[K:k]=p^{2}$. 

\begin{thm}[{Theorem \ref{thm:waps}}]\label{mth3}
Let $p>2$ be an odd prime number, and $k$ a global field. Consider a finite separable field extension $K/k$ of degree $p^{2}$ with Galois closure $\widetilde{K}/k$. Put $G:=\Gal(\widetilde{K}/k)$ and $H:=\Gal(\widetilde{K}/K)$. 
\begin{enumerate}
\item If $A_{k}(T_{K/k})\neq 1$, then \emph{($\ast$)} in Theorem \ref{mth1} is valid. 
\item If \emph{($\ast$)} holds, then there is an isomorphism
\begin{equation*}
A_{k}(T_{K/k})\cong 
\begin{cases}
\Z/p &\text{if a decomposition group of $\widetilde{K}/k$ contains $(C_{p})^{2}$; }\\
1 &\text{otherwise. }
\end{cases}
\end{equation*}
\end{enumerate}
\end{thm}

In the following, we sketch our proof of Theorems \ref{mth1} and \ref{mth2}. For a finite group $G$ and a subgroup $H$ of $G$, a $G$-lattice $J_{G/H}$ is defined by the exact sequence
\begin{equation*}
0\rightarrow \Z \rightarrow \Ind_{H}^{G}\Z \rightarrow J_{G/H} \rightarrow 0. 
\end{equation*}
Take an admissible set $\cD$ of subgroups in $G$ in the sense of \cite{Oki2025a}, and put
\begin{equation*}
\Sha_{\cD}^{2}(G,J_{G/H}):=\Ker\left(H^{2}(G,J_{G/H})\rightarrow \bigoplus_{D\in \cD}H^{2}(D,J_{G/H})\right). 
\end{equation*}
Then, Theorems \ref{mth1} and \ref{mth2} are reduced to studies on $\Sha_{\cD}^{2}(G,J_{G/H})$ in the case $(G:H)=p^{2}$. Moreover, to discuss Theorem \ref{mth2}, we only need the case that $\cD$ is the set of all cyclic subgroups $\cC_{G}$ in $G$. See Section \ref{sect:pfch} for details. 

In the following, we write $\Sha_{\cD}^{2}(G,J_{G/H})$ for $\Sha_{\omega}^{2}(G,J_{G/H})$ if $\cD=\cC_{G}$. 

\vspace{5pt}
\textbf{Case 1: $G$ is a $p$-group. }
In this case, the exponent of $G$, the least common multiple of the orders of all elements of $G$, is $p$ or $p^{2}$. See Section \ref{sect:trgp}. 

\emph{Case 1-1: $G$ has exponent $p^{2}$. }
It suffices to prove $\Sha_{\omega}^{2}(G,J_{G/H})=0$. To accomplish this, we use the same argument as \cite{Oki2025a} as follows. Let $N$ be a maximal subgroup of $G$ containing $H$. Note that it is unique if $\#G\geq p^{3}$. Then, there is an exact sequence of $G$-lattices
\begin{equation*}
0\rightarrow J_{G/N}\rightarrow J_{G/H}\rightarrow \Ind_{N}^{G}J_{N/H}\rightarrow 0. 
\end{equation*}
Taking cohomology, we obtain an exact sequence
\begin{equation*}
0\rightarrow H^{1}(N,J_{N/H})\rightarrow H^{2}(G,J_{G/N})\xrightarrow{i^{*}} H^{2}(G,J_{G/H})\rightarrow H^{2}(N,J_{N/H}). 
\end{equation*}
We denote by $\Sel_{H}^{2}(G,J_{G/N})$ the preimage of $\Sha_{\omega}^{2}(G,J_{G/H})$ under $i^{*}$. Then one has an exact sequence
\begin{equation*}
0\rightarrow H^{1}(N,J_{N/H})\rightarrow \Sel_{H,\cD}^{2}(G,J_{G/N})\xrightarrow{i^{*}} \Sha_{\cD}^{2}(G,J_{G/H})\rightarrow 0. 
\end{equation*}
On the other hand, there is a commutative diagram
\begin{equation*}
\xymatrix{
\Sha_{\omega}^{2}(G,J_{G/H})\ar@{^{(}->}[r]&H^{2}(G,J_{G/H})\ar[r]\ar[d]^{\delta}&H^{2}(N,J_{N/H})\ar[d]\\
&H^{3}(G,\Z)\ar[r]^{\Res_{G/N}}&H^{3}(N,\Z). 
}
\end{equation*}
One can prove that the bottom horizontal map $\Res_{G/N}$ is injective. It follows from $H^{3}(G,\Z)=0$ if $\#G\leq p^{3}$. In the case $\#G=p^{p+1}$, the desired injectivity follows from Tahara's theorem (\cite{Tahara1972}). In the remaining cases, we reduce the assertion to the vanishing of the \emph{Bogomolov multipler} of $G$ and apply the result of Fern{\'a}ndez-Alcober and Jezernick (\cite{FernandezAlcober2020}). Therefore, the study of $\Sel_{H}^{2}(G,J_{G/N})$ is reduced to determine the preimage $\sS_{H}^{2}(G,\Ind_{N}^{G}\Z)$ of $\Sha_{\omega}^{2}(G,J_{G/H})$ under the homomorphism $H^{2}(N,\Z)\rightarrow H^{2}(G,J_{G/N})$. However, an explicit computation using Mackey's decomposition gives an isomorphism $\sS_{H}^{2}(G,\Ind_{N}^{G}\Z)\cong (\Z/p)^{2}$. This implies 
\begin{equation*}
\Sel_{H}^{2}(G,J_{G/N})\cong \Z/p, 
\end{equation*}
and in particular, we obtain $\Sha_{\omega}^{2}(G,J_{G/H})=0$. 
\begin{rem}
If $\#G=p^{3}$, then $\Sha_{\omega}^{2}(G,J_{G/H})=0$ also follows from \cite[Theorem 1.2]{Hoshi2025a}. However, the proof is completely different from ours. 
\end{rem}

\emph{Case 1-2: $G$ has exponent $p$. }
Let $N$ be as in Case 1-1. We should consider $\Sha_{\cD}^{2}(G,J_{G/H})$ for any admissible set of subgroups $\cD$ of $G$. However, it is impossible to construct a direct analogue of the argument in Case 1-1. The crucial difference between Cases 1-1 and 1-2 is that the restriction map $H^{3}(G,\Z)\rightarrow H^{3}(E,\Z)$ may not be injective. In fact, if $G$ has order $p^{3}$, then we have the following: 
\begin{equation*}
H^{3}(G,\Z)\cong (\Z/p)^{\oplus 2},\quad H^{3}(N,\Z)\cong \Z/p. 
\end{equation*}
To overcome this problem, we consider a $p$-group $\widetilde{G}$ that is an extension of $G$. Let $\widetilde{G}$ be the wreath product $(C_{p})^{p}\rtimes C_{p}$. Then one can construct a surjection
\begin{equation*}
\pi \colon \widetilde{G}\twoheadrightarrow G. 
\end{equation*}
See Lemma \ref{lem:pmpr} (ii). Put $\widetilde{H}:=\pi^{-1}(H)$. Then $\pi$ induces an isomorphism
\begin{equation*}
\Sha_{\cD}^{2}(G,J_{G/H})\cong \Sha_{\widetilde{\cD}}^{2}(\widetilde{G},J_{\widetilde{G}/\widetilde{H}}),
\end{equation*}
where $\widetilde{\cD}$ is an admissible set of subgroups in $\widetilde{G}$ that satisfies $\{\pi(D)\mid D\in \cD\}=\cD$. Furthermore, if we set $\widetilde{N}:=\pi^{-1}(N)$, then the restriction map $H^{3}(\widetilde{G},\Z)\rightarrow H^{3}(\widetilde{N},\Z)$ is injective. This is already mentioned in Case 1-1. Therefore, we can apply the same method as Case 1-1 by considering the exact sequence of $\widetilde{G}$-lattices
\begin{equation*}
0\rightarrow J_{\widetilde{G}/\widetilde{N}}\rightarrow J_{\widetilde{G}/\widetilde{H}}\rightarrow \Ind_{\widetilde{N}}^{\widetilde{G}}J_{\widetilde{N}/\widetilde{H}}\rightarrow 0. 
\end{equation*}

\begin{rem}
If $\#G=p^{3}$, then the structure $\Sha_{\omega}^{2}(G,J_{G/H})$ is also obtained by \cite[Theorem 1.6 (2)]{Hoshi2025b}. In particular, we provide another proof of the previous work. 
\end{rem}

\textbf{Case 2: General case. }
Let $P$ be a $p$-Sylow subgroup of $G$. Denote by $P'_{2}$ the Heisenberg group of order $p^{3}$. If $P$ is not isomorphic to $(C_{p})^{2}$ or $P'_{2}$, then we obtain $\Sha_{\omega}^{2}(G,J_{G/H})=0$ by Case 1. Hence, we may assume that $P \cong (C_{p})^{2}$ or $P\cong P'_{2}$ is valid. In this case, one can prove that at least one of the following is satisfied: 
\begin{enumerate}
\item[(a)] there exist a subgroup $G'$ of $G$ with $(G:G')\leq 2$ and transitive groups $G_{1}$ and $G_{2}$ of degree $p$ with $\#G_{1}>p$ or $\#G_{2}>p$ such that $G_{1}\times G_{2}<G'<N_{\fS_{p}}(G_{1})\times N_{\fS_{p}}(G_{2})$; 
\item[(b)] there is a subgroup $G'$ of $\GL_{2}(\Fp)$ such that $G\cong (C_{p})^{2}\rtimes G'$ and $H\cong \{1\}\rtimes G'$. 
\end{enumerate}
This follows from the results of Dobson and Witte (\cite{Dobson2002}), which gives a partial classification of the possibility of $G$. 

\emph{Case 2-1: \emph{(a)} holds. }
In this case, there is an isomorphism $P \cong (C_{p})^{2}$. It suffices to discuss the case $G=G_{1}\times G_{2}$. Then the $p$-primary torsion part of $H^{2}(H,\Z)$ is trivial since $\#H\notin p\Z$. Moreover, Tahara's theorem (\cite{Tahara1972}) implies that the $p$-primary torsion part of $H^{3}(G,\Z)$ is also trivial. This implies $\Sha_{\omega}^{2}(G,J_{G/H})=0$ since the left-hand side is a $p$-abelian group. 

\emph{Case 2-2: \emph{(b)} holds. }
Let $P$ be a $p$-Sylow subgroup of $G$. Take a subgroup $N$ of order $p^{2}$ in $P$ so that $N\cap H=\{1\}$. Then, one can construct a commutative diagram
\begin{equation*}
\xymatrix{
\Sha_{\omega}^{2}(G,J_{G/H})\ar[r]\ar@{^{(}->}[d]^{\Res_{G/P}}& H^{3}(N,\Z)^{G^{\dagger}}\ar@{^{(}->}[d]^{x\mapsto x}\\
\Sha_{\omega}^{2}(P,J_{P/(P\cap H)})\ar@{^{(}->}[r]& H^{3}(N,\Z)\cong \Z/p. }
\end{equation*}
Moreover, it is known that $H^{2}(N,\Z)^{G'}\neq 0$ holds if and only if $G'<\SL_{2}(\Fp)$. Hence, we have $\Sha_{\omega}^{2}(G,J_{G/H})=0$ if $G'$ is not contained in $\SL_{2}(\Fp)$. On the other hand, in the case $G'<\SL_{2}(\Fp)$, it suffices to prove $\Sha_{\omega}^{2}(G,J_{G/H})\cong \Z/p$. We can reduce to the case $G'=\SL_{2}(\Fp)$. Then, there is a generalized representation group of $G$ of the form
\begin{equation}\label{eq:grsk}
1\rightarrow Z(P'_{2})\rightarrow P'_{2}\rtimes_{\widetilde{\varphi}} G' \rightarrow G \rightarrow 1,
\end{equation}
where $\widetilde{\varphi}$ is a lift of the standard standard representation of $G'$ that acts trivially on the center of $P'_{2}$. Our construction of $\widetilde{\varphi}$ is based on the result of Winter (\cite{Winter1972}), which studies the structure of the outer isomorphism group of $P'_{2}$. Applying a refinement of the theory of Drakokhrust and Macedo--Newton to \eqref{eq:grsk}, we obtain the desired isomorphism. Finally, the condition for the vanishing of $\Sha_{\cD}^{2}(G,J_{G/H})$ is reduced to Case 1-2. 

\begin{oftp}
In Section \ref{sect:trgp}, we rephrase the result of Dobson--Witte, which studies the structure of transitive group of prime squared degree. In Section \ref{sect:prlm}, we enumerate some previous results on cohomology groups and Chevalley modules. Sections \ref{sect:csha} and \ref{sect:cshg} are the technical heart of this paper, which determine the structure of $\Sha_{\cD}^{2}(G,J_{G/H})$ in the case that $(G:H)=p^{2}$ with $p$ an odd prime number. More precisely, the former half treats the case where $G$ is a $p$-group, and the latter half considers $\Sha_{\cD}^{2}(G,J_{G/H})$ in general. Finally, we give proofs of main theorems in Section \ref{sect:pfch}. 
\end{oftp}

\begin{ack}
This work is carried out with the support of KAKENHI Grant Number JP24K16884.
\end{ack}

\begin{nota}
Let $G$ be a finite group. 
\begin{itemize}
\item For $g,h\in G$, put $[g,h]:=g^{-1}hgh^{-1}$. 
\item Consider subgroups $H$ and $H'$ of $G$. Then, we denote by $[H,H']$ the subgroup of $G$ generated by $[h,h']$ for any $h\in H$ and $h'\in H'$. Moreover, we write $H^{\der}$ for $[H,H]$. 
\item Denote by $Z(G)$ the center of $G$, and put $G^{\ad}:=G/Z(G)$. 
\item For a finite group $G$ and its subgroup $H$, let
\begin{equation*}
N^{G}(H):=\bigcap_{g\in G}gHg^{-1}. 
\end{equation*}
Note that is called the \emph{normal core} of $H$ in $G$. This notation is also used in \cite{Hasegawa2025} and \cite{Oki2025a}. 
\item Set $G^{\vee}:=\Hom(G,\Q/\Z)$. It is the Pontryagin dual if $G$ is abelian. For any normal subgroup $N$ of $G$, the natural surjection $G\twoheadrightarrow G/N$ induces an isomorphism
\begin{equation*}
(G/N)^{\vee}\xrightarrow{\cong} \{f\in G^{\vee}\mid f\!\mid_{N}=0\}. 
\end{equation*}
We regard $(G/N)^{\vee}$ as a subgroup of $G^{\vee}$ by the above isomorphism. In particular, one has $G^{\vee}=(G/G^{\der})^{\vee}$. 
\item For each $g\in G$, let $\Ad(g)\colon G\rightarrow G;\,h\mapsto ghg^{-1}$, which is an automorphism of $G$. 
\end{itemize}
\end{nota}

\section{Transitive groups of prime squared degree}\label{sect:trgp}

Let $n$ be a positive integer, and denote by $\fS_{n}$ the symmetric group of $n$ letters. Recall that a \emph{transitive group of degree $n$} (or, a \emph{transitive subgroup of $\fS_{n}$}) is a subgroup of $\fS_{n}$ of which the action on the set $\Z/n$ is transitive. Here, we mainly discuss the case where $n=p^{2}$, where $p$ is a prime number. Moreover, we recall some facts on transitive groups of prime power degree. 

\begin{lem}[{\cite[Theorem 3.2]{Hasegawa2020}}]\label{lem:sytr}
Let $p$ be a prime number, and $n$ a positive integer. Consider a transitive group $G$ of degree $p^{n}$ with its one point stabilizer $H$. Then, a $p$-Sylow subgroup of $G_{p}$ of $G$ is a transitive group of degree $p^{n}$, and $H\cap G_{p}$ is a one point stabilizer in $G_{p}$. 
\end{lem}

For a prime number $p$, let $\ord_{p}\colon \Q^{\times}\rightarrow \Z$ be the homomorphism defined as follows for each prime number $\ell$: 
\begin{equation*}
\ord_{p}(\ell)=
\begin{cases}
1&\text{if }\ell=p;\\
0&\text{if }\ell \neq p. 
\end{cases}
\end{equation*}

\begin{lem}\label{lem:sydp}
Let $p$ be a prime number, and $G$ a transitive group of degree $p$. Then, a $p$-Sylow subgroup of $G$ is cyclic of order $p$. 
\end{lem}

\begin{proof}
This follows from $\ord_{p}(\#\fS_{p})=1$. 
\end{proof}

\begin{lem}\label{lem:nadp}
Let $p$ be a prime number, and $G$ a transitive group of degree $p$. Then, $\#G^{\ab}$ is coprime to $p$ if and only if $G\cong C_{p}$. 
\end{lem}

\begin{proof}
It is clear that $\#G^{\ab}\in p\Z$ if $G\cong C_{p}$. Conversely, assume that $\#G$ is a multiple of $p$. Then, $\#G^{\der}$ is coprime to $p$ by Lemma \ref{lem:sydp}. Hence, we have $G^{\der}=\{1\}$, that is, $G$ is abelian. This implies $G\cong C_{p}$ as desired. 
\end{proof}

\begin{rem}
A complete classification of transitive groups of prime degree is known. See \cite[p.~99, l.~15--34]{Dixon1996} for example. Note that we do not use the above result in this paper. 
\end{rem}

Let $p$ be a prime number. We recall some notation in \cite[\S 3]{Dobson2002} to discuss transitive groups of degree $p^{2}$. 
\begin{itemize}
\item For each $i_{0}\in \{0,\ldots,p-1\}$, put
\begin{equation*}
z_{i_{0}}\colon \Fp \times \Fp \rightarrow \Fp \times \Fp;\,(i,j)\mapsto 
\begin{cases}
(i_{0},j+1)&\text{if }i=i_{0};\\
(i,j)&\text{if }i\neq i_{0}. 
\end{cases}
\end{equation*}
\item For $n\in \{1,\ldots,p\}$, let
\begin{equation*}
\gamma_{n}:=z_{0}^{a_{p-n,0}}\cdots z_{p-1}^{a_{p-n,p-1}},
\end{equation*}
where
\begin{equation*}
a_{i,j}:=
\begin{cases}
(-1)^{i-j}\begin{pmatrix}i\\j\end{pmatrix}&\text{if }i\geq j; \\
0&\text{if }i<j. 
\end{cases}
\end{equation*}
Note that $a_{i+1,j}=a_{i,j}-a_{i,j-1}$. 
\item We define $\rho_{1}$ and $\rho_{2}$ as follows: 
\begin{align*}
\rho_{1}\colon \Fp \times \Fp \rightarrow \Fp \times \Fp&;\,(i,j)\mapsto (i,j+1);\\
\rho_{2}\colon \Fp \times \Fp \rightarrow \Fp \times \Fp&;\,(i,j)\mapsto (i+1,j). 
\end{align*}
\item Put $\tau:=z_{0}\rho_{2}$. 
\item For each $\beta \in \Fpt$, set
\begin{align*}
\overline{\beta}\colon \Fp \times \Fp \rightarrow \Fp \times \Fp&;\,(i,j)\mapsto (\beta i,j);\\
\widetilde{\beta}\colon \Fp \times \Fp \rightarrow \Fp \times \Fp&;\,(i,j)\mapsto (i,\beta j). 
\end{align*}
\end{itemize}
We regard them as an element of $\fS_{p}$ by the bijection
\begin{equation*}
\Z/p^{2} \xrightarrow{\cong} \Fp \times \Fp;\,a+pb\mapsto (a\bmod p,b\bmod p)\ (a,b\in \{0,\ldots,p-1\}). 
\end{equation*}

Now, for $n\in \{1,\ldots,p\}$, put
\begin{equation*}
P_{n}:=\langle \tau,\gamma_{n}\rangle,\quad P'_{n}:=\langle \rho_{1},\rho_{2},\gamma_{n}\rangle. 
\end{equation*}
These are transitive group of degree $p^{2}$. We remark that $\#P_{n}=\#P'_{n}=p^{n+1}$. Moreover, one has $P_{p}=P'_{p}$, and it is a $p$-Sylow subgroup of $\fS_{p^{2}}$. 

\begin{dfn}
For each $n\in \{1,\ldots,p\}$, put $\eta_{n}:=\rho_{2}\gamma_{n}\rho_{2}^{-1}$. 
\end{dfn}

\begin{lem}\label{lem:elst}
\begin{enumerate}
\item We have $\tau^{p}=\rho_{1}=\eta_{1}$ and $\eta_{p}=z_{1}$. 
\item For any $n,n'\in \{1,\ldots,p\}$, one has $[z_{n},z_{n'}]=[z_{n},\gamma_{n'}]=[z_{n},\eta_{n'}]=[\eta_{n},\eta_{n'}]=1$. 
\item For each $n\in \{1,\ldots,p\}$ and $i\in \{0,\ldots,p-1\}$, we have $[\eta_{n},\rho_{2}]=[\eta_{n},z_{i}\rho_{2}]=\eta_{n-1}$. 
\item For $n\in \{1,\ldots,p\}$ and $i\in \{0,\ldots,p-1\}$, we have
\begin{equation*}
\prod_{j=0}^{p-1}(z_{i}\rho_{2})^{j}\eta_{n}(z_{i}\rho_{2})^{-j}=
\begin{cases}
1&\text{if }n<p;\\
\eta_{1}&\text{if }n=p. 
\end{cases}
\end{equation*}
\item We have $\eta_{2}\rho_{2}\eta_{2}^{-1}=\rho_{1}^{-1}\rho_{2}$. 
\item For any $\beta \in \Fpt$, we have the following: 
\begin{equation*}
\overline{\beta}\rho_{1}\overline{\beta}^{-1}=\rho_{1},\quad \overline{\beta}\rho_{2}\overline{\beta}^{-1}=\rho_{2}^{\beta},\quad
\widetilde{\beta}\rho_{1}\widetilde{\beta}^{-1}=\rho_{1}^{\beta},\quad \widetilde{\beta}\rho_{2}\widetilde{\beta}^{-1}=\rho_{2}.
\end{equation*}
\end{enumerate}
\end{lem}

\begin{proof}
These follow from direct computation. 
\end{proof}

\begin{cor}
For every $n\in \{1,\ldots,p\}$, we have $P_{n}=\langle \tau,\eta_{n}\rangle$ and $P'_{n}=\langle \rho_{1},\rho_{2},\eta_{n}\rangle$. 
\end{cor}

\begin{proof}
The equality $P'_{n}=\langle \rho_{1},\rho_{2},\eta_{n}\rangle$ follows from the definition of $\eta_{n}$. On the other hand, Lemma \ref{lem:elst} (i) gives $P_{n}=\langle \tau,\eta_{n}\rangle$ as desired. 
\end{proof}

\begin{dfn}
\begin{enumerate}
\item For $n\in \{0,\ldots,p\}$, put
\begin{equation*}
N_{n}:=
\begin{cases}
\{1\}&\text{if }n=0;\\
\langle \eta_{1},\ldots,\eta_{n}\rangle&\text{if }n\geq 1. 
\end{cases}
\end{equation*}
\item For $n\in \{1,\ldots,p\}$, put
\begin{equation*}
H_{n}:=
\begin{cases}
\{1\}&\text{if }n=1;\\
\langle \eta_{2},\ldots,\eta_{n}\rangle&\text{if }n\geq 2. 
\end{cases}
\end{equation*}
\end{enumerate}
\end{dfn}

\begin{lem}\label{lem:enel}
Let $n\in \{1,\ldots,p\}$. 
\begin{enumerate}
\item The group $N_{n}$ is elementary abelian of order $p^{n}$ containing $H_{n}$. 
\item The group $N_{n}$ is maximal in $P_{n}$ and $P'_{n}$. In particular, we have $N_{n}=P_{n}\cap P'_{n}$. 
\item We have $P_{n}^{\der}=(P'_{n})^{\der}=N_{n-1}$, and $[N_{n'},P_{n}]=[N_{n'},P'_{n}]=N_{n'-1}$ for any $n'\in \{1,\ldots,n\}$. 
\item The group $H_{n}$ is the stabilizer of $(0,0)$ in $P_{n}$ and $P'_{n}$. 
\end{enumerate}
\end{lem}

\begin{proof}
(i): By Lemma \ref{lem:elst} (ii), we obtain that $N_{n}$ is abelian. Moreover, by definition, $\eta_{i}$ has order $p$ and $\langle \eta_{1},\ldots,\eta_{n'-1}\rangle \cap \langle \eta_{n'}\rangle=\{1\}$. Hence, $N_{n}$ is elementary abelian of order $p^{n}$. On the other hand, the inclusion $H_{n}<N_{n}$ follows from Lemma \ref{lem:elst} (iii). 

(ii): By Lemma \ref{lem:elst} (iii), we have $N_{n}<P_{n}$. Hence, it is maximal since $\#N_{n}=p^{n}=p^{-1}\#P_{n}$. 

(iii): This follows from Lemma \ref{lem:elst} (iii). 

(iv): By definition, $H_{n}$ is a subgroup of $P_{n}$ and $P'_{n}$ that stablizes $(0,0)\in \Fp \times \Fp$ in $P_{n}$. Hence, the assertion holds since $P_{n}$ and $P'_{n}$ act on $\Z/p^{2}$ transitively. 
\end{proof}

\begin{lem}\label{lem:pmpr}
\begin{enumerate}
\item For each $n\in \{1,\ldots,p\}$, we have $P'_{n}=N_{n}\rtimes \langle \rho_{2}\rangle$. 
\item We regard $N_{1},\ldots,N_{p}$ as $\Fp[\rho_{2}]$-modules. Then, for any $n\in \{1,\ldots,p\}$, one has a commutative diagram
\begin{equation*}
\xymatrix{
\F_{p}[\rho_{2}]\ar[r]^{\cong}\ar[d]_{\rho_{2} \mapsto \rho_{2}\bmod (\rho_{2}-1)^{n}}&N_{p}\ar[d]^{q_{n}}\\
\Fp[\rho_{2}]/(\rho_{2}-1)^{n}\ar[r]^{\hspace{30pt}\cong}& N_{n}. }
\end{equation*}
Here, $q_{n}$ is defined as follows: 
\begin{equation*}
q_{n}(\eta_{n'})=
\begin{cases}
\eta_{n'-(p-n)}&\text{if }n'\geq p-n;\\
1&\text{if }n'<p-n. 
\end{cases}
\end{equation*}
Moreover, the above diagram induces an exact sequence
\begin{equation}\label{eq:ppns}
1\rightarrow N_{p-n}\rightarrow P_{p}\xrightarrow{\pi_{n}} P'_{n}\rightarrow 1. 
\end{equation}
\end{enumerate}
\end{lem}

\begin{proof}
(i): This follows from Lemma \ref{lem:elst} (iii). 

(ii): Let $n'\in \{1,\ldots,p\}$. By the Frobenius reciprocity, the canonical injection 
\begin{equation*}
\Fp\hookrightarrow N_{n'};\,1\bmod p\mapsto \gamma_{n'}
\end{equation*}
induces a surjective homomorphism of $\Fp[\rho_{2}]$-modules 
\begin{equation*}
\Fp[\rho_{2}]\twoheadrightarrow N_{n'}
\end{equation*}
This induces the isomorphism $\psi_{n'} \colon \Fp[\rho_{2}]/(\rho_{2}-1)^{n'}\xrightarrow{\cong} N_{n}$, and one has the desired commutative diagram. Moreover, the above isomorphisms for $n'=n$ and $n'=p$ and the identity map on $\langle \rho_{2}\rangle$ induce a surjective homomorphism
\begin{equation*}
\pi_{n} \colon P_{p}\twoheadrightarrow P'_{n}. 
\end{equation*}
On the other hand, we have $\ker(\pi_{n})=\Ker(q_{n})$ by definition. It is clear that
\begin{equation*}
\Ker(q_{n})=\langle \eta_{1},\ldots,\eta_{p-n}\rangle=N_{p-n}, 
\end{equation*}
and therefore we obtain \eqref{eq:ppns}. 
\end{proof}

\begin{lem}\label{lem:pnpr}
Let $n\in \{1,\ldots,p-1\}$. 
\begin{enumerate}
\item The following holds: 
\begin{equation*}
Z(P_{n})=
\begin{cases}
P_{1}&\text{if }n=1;\\
N_{1}&\text{if }n\geq 2. 
\end{cases}
\end{equation*}
\item Assume $n\geq 2$. For any $g\in P_{n}\setminus N_{n}$, one has $g^{p}\in Z(P_{n})\setminus \{1\}$. In particular, $P_{n}$ has exponent $p^{2}$. 
\item All non-cyclic abelian subgroups of $P_{n}$ are contained in $N_{n}$. 
\end{enumerate}
\end{lem}

\begin{proof}
(i): It is clear in the case $n=1$. In the following, assume $n\geq 2$. By definition, we have $N_{1}<Z(P_{n})$. On the other hand, Lemma \ref{lem:elst} (ii) and (iii) imply
\begin{equation}\label{eq:cmgm}
[\eta_{2}^{m_{2}}\cdots \eta_{n}^{m_{n}},\eta_{2}^{m'_{2}}\cdots \eta_{n}^{m'_{n}}z_{p-1}\rho_{2}]=\eta_{1}^{-m_{2}}\cdots \eta_{n-1}^{-m_{n-1}}
\end{equation}
for any $m_{2},\ldots,m_{n},m'_{2},\ldots,m'_{n}\in \{0,\ldots,p-1\}$. Hence, \eqref{eq:cmgm} for $m_{2}=1$ implies $Z(P_{n})<N_{n}$. Furthermore, \eqref{eq:cmgm} for $m'_{1}=\cdots=m'_{n}=0$ implies $Z(P_{n})<N_{1}$. Therefore, the proof is complete. 

(ii): Take $g\in P_{n}\setminus N_{n}$. Then, there exist $m_{1},\ldots,m_{n},m'\in \{0,\ldots,p-1\}$ such that
\begin{equation*}
g=\eta_{1}^{m_{1}}\cdots \eta_{n}^{m_{n}}\tau^{m'}. 
\end{equation*}
Replacing $g$ with its suitable power, we may assume $m'=1$. Then, we have
\begin{equation*}
g^{p}=\left(\prod_{i=0}^{p-1}\tau^{i}\eta_{1}^{m_{1}}\cdots \eta_{n}^{m_{n}}\tau^{-i}\right)\tau^{p}
\end{equation*}
Moreover, we have $[\eta_{i},\eta_{j}]=1$ by Lemma \ref{lem:elst} (ii), and hence
\begin{equation}\label{eq:cpgp}
\prod_{i=0}^{p-1}\tau^{i}\eta_{1}^{m_{1}}\cdots \eta_{n}^{m_{n}}\tau^{-i}=\prod_{n'=1}^{n}\left(\prod_{i=0}^{p-1}(z_{p-1}\rho_{2})^{i}\eta_{n'}(z_{p-1}\rho_{2})^{-i}\right)^{m_{n'}}. 
\end{equation}
The right-hand hand of \eqref{eq:cpgp} equals $1$, which is a consequence of Lemma \ref{lem:elst} (iv). Combining this fact with Lemma \ref{lem:elst} (i), we obtain $g^{p}=\eta_{1}$. Hence, the assertion follows from (i). 

(iii): It is clear in the case $n=1$. In the following, assume $n\geq 2$. Pick a subgroup $D$ of $P_{n}$ that is not contained in $N_{n}$. It suffices to prove that $D$ is not abelian. Take $g\in D\setminus N_{n}$, which is possible by assumption. Then, we have $\langle g \rangle \cap N_{n}=Z(P_{n})$ by (ii). On the other hand, since $D$ is not cyclic, one can pick an element $\gamma$ of $D\cap (N_{n}\setminus Z(P_{n}))$. Now, we have $[\gamma,g]\neq 1$ by Lemma \ref{lem:elst} (iii), which implies that $D$ is not abelian. 
\end{proof}

\begin{lem}\label{lem:pppr}
Let $n\in \{1,\ldots,p-1\}$. 
\begin{enumerate}
\item The following holds: 
\begin{equation*}
Z(P'_{n})=
\begin{cases}
P'_{n}&\text{if }n=1;\\
\langle \eta_{1}\rangle&\text{if }n\geq 2. 
\end{cases}
\end{equation*}
\item The group $P'_{n}$ has exponent $p$. 
\end{enumerate}
\end{lem}

\begin{proof}
(i): The same as Lemma \ref{lem:pnpr} (i). 

(ii): Take $g\in P'_{n}$, then we have
\begin{equation*}
g=\eta_{1}^{m_{1}}\cdots \eta_{n}^{m_{n}}\rho_{2}^{m'}
\end{equation*}
for some $m_{1},\ldots,m_{n},m'\in \{0,\ldots,p-1\}$. We may assume $m'=1$. Then, we have
\begin{equation*}
g^{p}=\left(\prod_{i=0}^{p-1}\rho_{2}^{i}\eta_{1}^{m_{1}}\cdots \eta_{n}^{m_{n}}\rho_{2}^{-i}\right)\rho_{2}^{p}=\prod_{n'=0}^{n}\left(\prod_{i=0}^{p-1}\rho_{2}^{i}\eta_{n'}\rho_{2}^{-i}\right)^{m_{n'}}\rho_{2}^{p}
\end{equation*}
by Lemma \ref{lem:elst} (ii). The right-hand side of the above equality equals $1$, which follows from Lemma \ref{lem:elst} (iv) and $\rho_{2}^{p}=1$. Therefore, the assertion holds. 
\end{proof}

Next, we give a relation between transitive groups of degree $p^{2}$ and their $p$-Sylow subgroups.

\begin{prop}
\label{prop:dwpg}
Let $p$ be a prime number, and $G$ a transitive group of degree $p^{2}$. Take a $p$-Sylow subgroup $P$ of $G$, and put $n:=\ord_{p}(\#P)-1$. Then there is $s\in \fS_{p^{2}}$ such that $sGs^{-1}$ contains $P_{n}$ or $P'_{n}$ as a $p$-Sylow subgroup. 
\end{prop}

\begin{proof}
This follows from \cite[p.~49, Theorem 9]{Dobson2002} and \cite[p.~50, l.~3--4, Remark]{Dobson2002}. We give a proof for reader's convenience. 

\emph{Case 1.~$G$ has exponent a multiple of $p^{2}$. }
In this case, there is an element of order $p^{2}$ in $G$, and it is a $p^{2}$-cycle in $\fS_{p^{2}}$. Hence, we obtain $\tau \in sGs^{-1}$ for some $s\in \fS_{p^{2}}$. Now, let $P$ be a $p$-Sylow subgroup of $sGs^{-1}$. Then we obtain $P=P_{n}$ by \cite[Theorem 9]{Dobson2002}. 

\emph{Case 2.~The remaining case. }
By Lemma \ref{lem:sytr}, we may assume that $G$ is a $p$-group. Take a one point stabilizer $H$ of $G$. We first prove the existence of a subgroup $P$ of order $p^{2}$ in $G$ that satisfies $P\cap H=\{1\}$. By definition, we have $N^{G}(H)=\{1\}$. This gives an equality $H\cap Z(G)=\{1\}$. Note that $Z(G)$ is non-trivial since $G$ is assumed to be a $p$-group. If $\#Z(G)\geq p^{2}$, then any subgroup of order $p^{2}$ in $Z(G)$ is a desired one. If $\#Z(G)=p$, then $HZ(G)/Z(G)$ is a maximal subgroup of $G/Z(G)$. On the other hand, the exponent of $G/Z(G)$ equals $p$. Consequently, we can pick a subgroup $P'$ of order $p$ in $G/Z(G)$ so that $P'\cap (HZ(G)/Z(G))=\{1\}$. Then, the preimage $P$ of $P'$ under the canonical surjection $G\twoheadrightarrow G/Z(G)$ satisfies the desired property. 

Since $G$ has exponent $p$, we have $P\cong (C_{p})^{2}$. Hence we obtain $P'_{1}<s_{1}Gs_{1}^{-1}$ for some $s_{1} \in \fS_{p}$. 
Pick a $p$-Sylow subgroup $P$ of $s_{1}Gs_{1}^{-1}$. Then, by \cite[Theorem 9]{Dobson2002}, we have $P=s_{2}^{-1}P'_{n}s_{2}$ for some $s_{2}\in \Aut(\Fp \times \Fp)\subset \fS_{p^{2}}$. Now, put $s:=s_{2}s_{1}$, then we have $P'_{n}<sGs^{-1}$ as desired. 
\end{proof}

In the following, we fix a primitive $(p-1)$-th root $\zeta_{p-1}$ of unity in $\Fp$, which is a generator of $\F_{p}^{\times}$. 

\begin{prop}\label{prop:dwrf}
Let $p$ be a prime number, and $G$ a transitive group of degree $p^{2}$. We further assume $n:=\ord_{p}(\#G)-1\in \{1,\ldots,p-1\}$ and that $G$ contains $P'_{n}$. Then, $G$ satisfies one of the following: 
\begin{itemize}
\item[(1)] there is a subgroup $G'$ of $\GL_{2}(\Fp)$ such that $G\cong (C_{p})^{2}\rtimes_{\varphi} G'$ and $H\cong \{1\}\rtimes_{\varphi} G'$, where $\varphi$ is induced by the standard representation of $\GL_{2}(\Fp)$ over $\Fp$; 
\item[(2)] $G<\fS_{p}\times \fS_{p}$; 
\item[(3)] there is a subgroup $G'$ of index $2$ in $G$ that satisfies \emph{(2)}; 
\item[(4)] $n\notin \{1,p-1\}$ and $P\triangleleft G$; 
\item[(5)] $n=p-1$ and $G=PL$, where $L$ satisfies $(G:L)=p^{p-1}$, $P\cap L=P'_{1}$ and 
\begin{equation*}
P'_{1}<L<\fS_{p}\times \langle \rho_{2},\widetilde{\zeta}_{p-1} \rangle. 
\end{equation*}
\end{itemize}
\end{prop}

Recall that the \emph{standard representation of $\GL_{2}(\Fp)$} is the $2$-dimensional $\Fp$-representation $V=\Fp^{\oplus 2}$ of $\GL_{2}(\Fp)$ defined as follows: 
\begin{equation*} 
\GL_{2}(\Fp)\times V \rightarrow V;\,
\left(\begin{pmatrix}a&b\\c&d\end{pmatrix},(x,y)\right)\mapsto (ax+cy,bx+dy). 
\end{equation*}

\begin{proof}
This follows from the proof of \cite[Theorem 4]{Dobson2002}. See also \cite[Lemma 9]{Dobson2002} in the case $n=p-1$. 
\end{proof}

\begin{prop}\label{prop:crdw}
Let $p>2$ be an odd prime number, and $G$ a transitive group of degree $p^{2}$. We further assume $n:=\ord_{p}(\#G)-1\in \{1,2\}$ and that $G$ contains $P'_{n}$. 
\begin{enumerate}
\item One of the following is satisfied: 
\begin{itemize}
\item[(a)] there exists a subgroup $G'$ of $G$ with $(G:G')\leq 2$ and transitive groups $G_{1}$ and $G_{2}$ of degree $p$ with $\#G_{1}>p$ or $\#G_{2}>p$ such that
\begin{equation*}
G_{1}\times G_{2}<G'<N_{\fS_{p}}(G_{1})\times N_{\fS_{p}}(G_{2}); 
\end{equation*}
\item[(b)] there is a subgroup $G^{\dagger}$ of $\GL_{2}(\Fp)$ such that 
\begin{equation*}
G\cong (C_{p})^{2}\rtimes_{\varphi} G^{\dagger}\quad H\cong \{1\}\rtimes_{\varphi} G^{\dagger}; 
\end{equation*}
where $\varphi$ is induced by the standard representation of $\GL_{2}(\Fp)$ over $\Fp$. 
\end{itemize}
\item Assume that $G$ satisfies \emph{(a)} and \emph{(b)}. Write $G\cong (C_{p})^{2}\rtimes_{\varphi} G^{\dagger}$, where $G^{\dagger}<\GL_{2}(\Fp)$. Then, $G^{\dagger}$ is not contained in $\SL_{2}(\Fp)$. 
\end{enumerate}
\end{prop}

\begin{proof}
(i): Put $P:=P'_{n}$, which a $p$-Sylow subgroup of $G$. 

\emph{Case 1.~\emph{(1)} holds. }
In this case, it is clear that (b) is satisfied. 

\emph{Case 2.~\emph{(2)} holds. }
By \cite[Lemma 1]{Dobson2002}, there exist transitive groups $G_{1}$ and $G_{2}$ of degree $p$ such that
\begin{equation*}
G_{1}\times G_{2}<G<N_{\fS_{p}}(G_{1})\times N_{\fS_{p}}(G_{2}). 
\end{equation*}
If $\#G_{1}>p$ or $\#G_{2}>p$, then it is clear that (a) is valid. If $\#G_{1}=\#G_{2}=p$, then $P$ is normal in $G$ since $C_{p}$ is a characteristic subgroup of $N_{\fS_{p}}(C_{p})$. Hence, this case is contained in the validity of (4). 

\emph{Case 3.~\emph{(3)} holds. }
Take transitive groups $G'_{1}$ and $G'_{2}$ of degree $p$ such that
\begin{equation*}
G'_{1}\times G'_{2}<G'<N_{\fS_{p}}(G'_{1})\times N_{\fS_{p}}(G'_{2}). 
\end{equation*}
If $\#G'_{1}>p$ or $\#G'_{2}>p$, then (a) is valid. If $\#G'_{1}=\#G'_{2}=p$, then $P$ is a characteristic subgroup of $G'$. Since  $p>2=(G:G')$, we obtain that $P$ is normal in $G$. Therefore, this case is contained in the validity of (4). 

\emph{Case 4.~\emph{(4)} holds. }
It suffices to prove that (b) is valid. By \cite[Lemma 6]{Dobson2002}, there is an isomorphism
\begin{equation*}
N_{\fS_{p^{2}}}(P)= 
\begin{cases}
P'_{1}\rtimes \GL_{2}(\Fp)&\text{if }P=P'_{1};\\
P'_{3}\rtimes \langle \overline{\zeta}_{p-1},\widetilde{\zeta}_{p-1} \rangle&\text{if }P=P'_{2}. 
\end{cases}
\end{equation*}
Hence, if $P=P'_{1}$, then the validity of (b) is clear. In the following, assume $P=P'_{2}$. Then, we have
\begin{equation*}
N_{\fS_{p^{2}}}(P)/P\cong P'_{3}/P'_{2}\rtimes \langle \overline{\zeta}_{p-1},\widetilde{\zeta}_{p-1} \rangle, 
\end{equation*}
where $\overline{\zeta}_{p-1}$ and $\widetilde{\zeta}_{p-1}$ act on $P'_{3}/P'_{2}$ as the identity map and $\gamma_{3}P'_{2}\mapsto \gamma_{3}^{\zeta_{p-1}}P'_{2}$, respectively. This induces an inclusion
\begin{equation*}
G/P<P'_{3}/P'_{2}\rtimes \langle \overline{\zeta}_{p-1},\widetilde{\zeta}_{p-1} \rangle.  
\end{equation*}
Since $P'_{3}/P'_{2}\cong C_{p}$ and $(G:P)$ is coprime to $p$, \cite[Lemma 6.3 (iii)]{Oki2025a} implies the existence of $s\in P'_{3}$ and $\overline{G}<\langle \overline{\zeta}_{p-1},\widetilde{\zeta}_{p-1} \rangle$ such that $sGs^{-1}/P=\{P'_{2}\}\rtimes \overline{G}$. In particular, we have 
\begin{equation*}
sGs^{-1}=P'_{2}\rtimes \overline{G}. 
\end{equation*}
Therefore, replacing $G$ with $sGs^{-1}$, we may assume $P'_{2}<G<P'_{2}\rtimes \langle \overline{\zeta}_{p-1},\widetilde{\zeta}_{p-1} \rangle$. It suffices to prove the assertion in the case $G=P'_{2}\rtimes \langle \overline{\zeta}_{p-1},\widetilde{\zeta}_{p-1} \rangle$. Put 
\begin{equation*}
G^{\dagger}:=\langle \eta_{2},\overline{\zeta}_{p-1},\widetilde{\zeta}_{p-1}\rangle. 
\end{equation*}
Then, it stabilizes $(0,0)$, and hence $P'_{1}\cap G^{\dagger}=\{1\}$. Moreover, we have $P'_{1}G^{\dagger}=G$ by definition. Now, recall that the following hold: 
\begin{equation}\label{eq:gdcj}
\begin{gathered}
\eta_{2}\rho_{1}\eta_{2}^{-1}=\rho_{1},\quad 
\eta_{2}\rho_{2}\eta_{2}^{-1}=\rho_{1}^{-1}\rho_{2},\\
\overline{\zeta}_{p-1}\rho_{1}\overline{\zeta}_{p-1}^{-1}=\rho_{1},\quad
\overline{\zeta}_{p-1}\rho_{2}\overline{\zeta}_{p-1}^{-1}=\rho_{2}^{\zeta_{p-1}},\\
\widetilde{\zeta}_{p-1}\rho_{1}\widetilde{\zeta}_{p-1}^{-1}=\rho_{1}^{\zeta_{p-1}},\quad
\widetilde{\zeta}_{p-1}\rho_{2}\widetilde{\zeta}_{p-1}^{-1}=\rho_{2}. 
\end{gathered}
\end{equation}
These are consequences of Lemma \ref{lem:elst} (ii), (iii), (vi). This implies that $G=P'_{1}\rtimes G^{\dagger}$. In particular, $G^{\dagger}$ is the stabilizer of $(0,0)$ in $G$. 

In the following, we prove that (b) is valid. Fix a generator $c$ of $C_{p}$. Consider the isomorphism
\begin{equation*}
\psi \colon P'_{1}\xrightarrow{\cong} (C_{p})^{2}
\end{equation*}
defined as $\psi(\rho_{1})=(c,1)$ and $\psi(\rho_{2})=(1,c)$. Then, the following hold by \eqref{eq:gdcj}: 
\begin{gather*}
\psi(\eta_{2}\rho_{1}\eta_{2}^{-1})=(c,1),\quad \psi(\eta_{2}\rho_{2}\eta_{2}^{-1})=(c^{-1},c)\\
\psi(\overline{\zeta}_{p-1}\rho_{1}\overline{\zeta}_{p-1}^{-1})=(c,1),\quad
\psi(\overline{\zeta}_{p-1}\rho_{2}\overline{\zeta}_{p-1}^{-1})=(1,c^{\zeta_{p-1}}),\\
\psi(\widetilde{\zeta}_{p-1}\rho_{1}\widetilde{\zeta}_{p-1}^{-1})=(c^{\zeta_{p-1}},1),\quad
\psi(\widetilde{\zeta}_{p-1}\rho_{2}\widetilde{\zeta}_{p-1}^{-1})=(1,c). 
\end{gather*}
Let $\omega$ be the homomorphism $G^{\dagger}\rightarrow \GL_{2}(\Fp)$ defined as follows: 
\begin{equation*}
\omega(\eta_{2})=\begin{pmatrix}1&-1\\0&1\end{pmatrix},\quad
\omega(\overline{\zeta}_{p-1})=\begin{pmatrix}1&0\\0&\zeta_{p-1}\end{pmatrix},\quad
\omega(\widetilde{\zeta}_{p-1})=\begin{pmatrix}\zeta_{p-1}&0\\0&1\end{pmatrix}. 
\end{equation*}
Then, $\psi$ and $\omega$ induce an injective homomorphism
\begin{equation*}
\Psi \colon G=P'_{1}\rtimes G^{\dagger}\hookrightarrow (C_{p})^{2}\rtimes_{\varphi}\GL_{2}(\Fp), 
\end{equation*}
where $\varphi$ is induced by the standard representation of $\GL_{2}(\Fp)$ over $\Fp$. Moreover, we have $\Psi(P'_{1})=(C_{p})^{2}\rtimes_{\varphi}\{1\}$ and $\Psi(G^{\dagger})<\{1\}\rtimes_{\varphi}\GL_{2}(\Fp)$. Hence, the proof is complete. 

\emph{Case 5.~\emph{(5)} holds. }
This case happens only when $p=3$. In particular, one has $P=P'_{2}$. By assumption, we have
\begin{equation*}
P'_{1}<L<P'_{1}\rtimes \langle \overline{-1},\widetilde{-1}\rangle. 
\end{equation*}
Hence, we may assume $L=P'_{1}\rtimes \langle \overline{-1},\widetilde{-1}\rangle$. Then, we have 
\begin{equation*}
G=PL=P'_{2}\rtimes \langle \overline{-1},\widetilde{-1}\rangle. 
\end{equation*}
Therefore, the same argument as Case 4 implies that (b) is valid. 

(ii): By replacing $G$ with its subgroup of index $2$ if necessary, we may assume
\begin{equation*}
G_{1}\times G_{2}<G<N_{\fS_{p}}(G_{1})\times N_{\fS_{p}}(G_{2}). 
\end{equation*}
Furthermore, by interchanging $G_{1}$ and $G_{2}$ if necessary, we may further assume $\#G_{1}>p$. By assumption, $P'_{1}$ is a unique $p$-Sylow subgroup in $G$. Hence, one has $G<P'_{1}\rtimes \langle \overline{\zeta}_{p-1},\widetilde{\zeta}_{p-1}\rangle$. Put $G^{\dagger}:=G\cap \langle \overline{\zeta}_{p-1},\widetilde{\zeta}_{p-1}\rangle$, which satisfies $G=P'_{1}\rtimes G^{\dagger}$. Then, we have $\overline{\beta}\in G$ for some $\beta \in \Fpt \setminus \{1\}$, which follows from $\#G_{1}>p$. Now, fix an injection
\begin{equation*}
\Psi \colon P'_{1}\rtimes \langle \overline{\zeta}_{p-1},\widetilde{\zeta}_{p-1}\rangle\hookrightarrow (C_{p})^{2}\rtimes_{\varphi}\GL_{2}(\Fp), 
\end{equation*}
where $\varphi$ is induced by the standard representation of $\GL_{2}(\Fp)$ over $\Fp$, that satisfies 
\begin{equation*}
\Psi(\langle \overline{\zeta}_{p-1},\widetilde{\zeta}_{p-1}\rangle)<\{1\}\rtimes_{\varphi}\GL_{2}(\Fp). 
\end{equation*}
Then, we have $\Psi(P'_{1})=(C_{p})^{2}\rtimes_{\varphi}\{1\}$ since $P'_{1}$ is normal in $P'_{1}\rtimes \langle \overline{\zeta}_{p-1},\widetilde{\zeta}_{p-1}\rangle$. Recall Lemma \ref{lem:elst} (vi) that $\overline{\beta}\rho_{1}\overline{\beta}=\rho_{1}$ and $\overline{\beta}\rho_{2}\overline{\beta}=\rho_{2}^{\beta}$ hold. Hence, we obtain $(\det \circ \Psi)(\overline{\beta})=\beta \neq 1$, which implies that $G^{\dagger}$ is not contained in $\SL_{2}(\Fp)$. 
\end{proof}

\section{Chevalley modules and their cohomology groups}\label{sect:prlm}

\subsection{General facts on cohomology}\label{sect:bsft}

Let $G$ be a finite group. A \emph{$G$-module} refers to a finitely generated abelian groups equipped with left actions of $G$. Moreover, a $G$-module which is torsion-free over $\Z$ is said to be a \emph{$G$-lattice}. On the other hand, for a subgroup $H$ of $G$ and an $H$-module $M$, put $\Ind_{H}^{G}M:=\Hom_{\Z[H]}(\Z[G],M)$. We define a left action of $G$ on $\Ind_{H}^{G}M$ by the map
\begin{equation*}
G \times \Ind_{H}^{G}M \rightarrow \Ind_{H}^{G}M; (g,\varphi) \mapsto [g'\mapsto \varphi(g'g)]. 
\end{equation*}

\begin{prop}[{Mackey decomposition; cf.~\cite[Section 7.3, Proposition 22]{Serre1977}}]\label{prop:mcky}
Let $G$ be a finite group, and $H$ and $D$ subgroups of $G$. Take a complete representative $R(D,H)$ of $D\backslash G/H$ in $G$. Consider an $H$-module $(M,\rho)$. Then the homomorphism of abelian groups
\begin{equation*}
\pi_{D,g}\colon \Ind_{H}^{G}M\rightarrow \Ind_{D\cap gHg^{-1}}^{D}M^{g};\varphi \mapsto [d\mapsto \varphi(g^{-1}d)]
\end{equation*}
for all $g\in R(D,H)$ induces an isomorphism of $D$-modules
\begin{equation*}
\Ind_{H}^{G}M \cong \bigoplus_{g\in R(D,H)}\Ind_{D\cap gHg^{-1}}^{D}M^{g}. 
\end{equation*}
Here $M^{g}$ is the abelian group $M$ equipped with the action of $gHg^{-1}$ defined by
\begin{equation*}
gHg^{-1}\rightarrow \Aut(M); h\mapsto (\rho \circ \Ad(g^{-1}))(h). 
\end{equation*}
\end{prop}

We need a description of the restriction map on the cohomology of $\Ind_{H}^{G}\Z$. 

\begin{lem}\label{lem:gpac}
Let $G$ be a finite group, and $M$ a $G$-module. For every subgroup $D$ of $G$, $g\in G$ and $j\in \Znn$, one has a commutative diagram
\begin{equation*}
\xymatrix@C=60pt{
H^{j}(G,M)\ar[r]^{\Res_{G/D}}\ar@{=}[d] & H^{j}(D,M)\ar[d]^{\cong}\\
H^{j}(G,M)\ar[r]^{\Res_{G/gDg^{-1}}\hspace{15pt}}& H^{j}(gDg^{-1},M). 
}
\end{equation*}
Here the right vertical homomorphism are induced by the inner automorphism $\Ad(g^{-1})$ on $G$ and the homomorphism $M\rightarrow M;x\mapsto gx$. 
\end{lem}

\begin{proof}
This is a special case of \cite[Lemma 2.5]{Oki2025a}. See also \cite[p.~46, Conjugation]{Neukirch2000}. 
\end{proof}

\begin{lem}[{\cite[Proposition 2.6]{Oki2025a}; cf.~\cite[p.~52]{Neukirch2000}}]\label{lem:htch}
Let $G$ be a finite group, and $H$ its subgroup. Then, there is a commutative diagram
\begin{equation*}
\xymatrix{
H^{2}(G,\Z)\ar[r]\ar[d]^{\cong}& H^{2}(G,\Ind_{H}^{G}\Z)\ar[d]^{\cong}\\
G^{\vee}\ar[r]^{f\mapsto f\mid_{H}}& H^{\vee}. }
\end{equation*}
Here, the top horizontal map is induced by the homomorphism of $G$-lattices
\begin{equation*}
\varepsilon_{G/H}\colon \Z \rightarrow \Ind_{H}^{G}\Z
\end{equation*}
defined by $\varepsilon_{G/H}(1)(g)=1$ for any $g\in G$. 
\end{lem}

\begin{lem}[{\cite[Proposition 2.8]{Oki2025a}; cf.~\cite[Section 6, pp.~310--311]{Platonov1994}}]\label{lem:rsch}
Let $G$ be a finite group, and $H$ and $D$ subgroups of $G$. Then the following diagram is commutative: 
\begin{equation*}
\xymatrix@C=90pt{
H^{2}(G,\Ind_{H}^{G}\Z)\ar[d]^{\cong}\ar[r]^{\Res_{G/D}}& H^{2}(D,\Ind_{H}^{G}\Z) \ar[d]^{\cong}\\
H^{\vee}\ar[r]^{f\mapsto ((f\circ \Ad(g^{-1}))\mid_{D\cap gHg^{-1}})_{g}\hspace{20mm}}& \bigoplus_{g\in R(D,H)}(D\cap gHg^{-1})^{\vee}. 
}
\end{equation*}
Here $R(D,H)$ is a complete representative of $D\backslash G/H$ in $G$. 
\end{lem}

Let $G$ be a finite group, and denote by $\cC_{G}$ the set of all cyclic subgroups in $G$. Let $M$ be a $G$-module, and $\cD$ a set of subgroups of $G$. Then, define
\begin{equation*}
\Sha_{\cD}^{2}(G,M):=\Ker\left(H^{2}(G,M)\xrightarrow{(\Res_{G/D})_{D\in \cD}} \bigoplus_{D\in \cD}H^{2}(D,M) \right). 
\end{equation*}
Here, $\Res_{G/D} \colon H^{2}(G,M)\rightarrow H^{2}(D,M)$ is the restriction map. Furthermore, we write $\Sha_{\omega}^{2}(G,M)$ for $\Sha_{\cC_{G}}^{2}(G,M)$. 

\vspace{6pt}
Recall \cite[Definition 2.9]{Oki2025a} that a set of subgroups $\cD$ of $G$ is said to be \emph{admissible} if
\begin{itemize}
\item $\cC_{G}\subset \cD$; and
\item $gDg^{-1}\in \cD$ for any $g\in G$ and $D\in \cD$. 
\end{itemize}

\begin{prop}[{\cite[Proposition 2.12]{Oki2025a}}]\label{prop:ifts}
Let $G$ be a finite group, and $N$ its normal subgroup. Take an admissible set of subgroups $\cD$ of $G$ and a $G/N$-module $M$ which is torsion-free as an abelian group. Then the inflation map $\Inf_{G/(G/N)}\colon H^{2}(G/N,M)\rightarrow H^{2}(G,M)$ induces an isomorphism
\begin{equation*}
\Sha_{{\cD}_{/N}}^{2}(G/N,M) \cong \Sha_{\cD}^{2}(G,M). 
\end{equation*}
Here $\cD_{/N}:=\{DN/N<G/N \mid D\in \cD\}$, which is an admissible set of subgroups in $G/N$. 
\end{prop}

For an abelian group $A$ and a prime number $p$, put
\begin{equation*}
A[p^{\infty}]:=\{a\in A\mid p^{m}a=0\text{ for some }m\in \Zpn\}. 
\end{equation*}

\begin{prop}[{\cite[Proposition 2.13 (ii)]{Oki2025a}}]\label{prop:shrs}
Let $G$ be a finite group, and $M$ a $G$-lattice. Take a subgroup $H$ of $G$ and a prime number $p$ that is coprime to $(G:H)$. Let $\cD$ be an admissible set of subgroups of $G$, and put $\cD_{\cap H}:=\{D\cap H<H\mid D\in \cD\}$, which is an admissible set of subgroups in $H$. Then, for any $j\in \Zpn$, the homomorphism
\begin{equation*}
\Res_{G/H}\colon \Sha_{\cD}^{2}(G,M)[p^{\infty}]\rightarrow \Sha_{\cD_{\cap H}}^{2}(H,M)
\end{equation*}
is injective. 
\end{prop}

\begin{lem}[{\cite[Corollary 2.19]{Oki2025a}}]\label{lem:shsh}
Let $G$ be a finite group, $H$ a subgroup of $G$. For a $H$-lattice $M$, there is an isomorphism
\begin{equation*}
\Sha_{\omega}^{2}(G,\Ind_{H}^{G}M)\cong \Sha_{\omega}^{2}(H,M). 
\end{equation*}
\end{lem}

\subsection{Chevalley modules}

Let $G$ be a finite group $G$, and $\cH$ a multiset of its subgroups. We define a $G$-lattice $J_{G/\cH}$ by the exact sequence
\begin{equation*}
0\rightarrow \Z \xrightarrow{\varepsilon_{G/H}^{\circ}} \Ind^{G}_{H}\Z \xrightarrow{j_{G/H}} J_{G/H} \rightarrow 0. 
\end{equation*}
Here, the homomorphism $\varepsilon_{G/H}^{\circ}$ is given by $1\mapsto [\sum_{g\in G}c_{g}g\mapsto \sum_{g\in G}c_{g}]$. The $G$-module $J_{G/H}$ is called the \emph{Chevalley module} of $G/H$. If $H=\{1\}$, we simply denote $J_{G/H}$ by $J_{G}$. 

\begin{prop}\label{prop:andg}
Let $G$ be a finite group, and $H$ its subgroup. Then, $\Sha_{\omega}^{2}(G,J_{G/H})$ is annihilated by $(G:H)$. 
\end{prop}

\begin{proof}
This is a special case of \cite[Corollary 5.4]{Oki2025a}. 
\end{proof}

Finally, we remind a powerful tool to compute $\Sha_{\omega}^{2}(G,J_{G/H})$ for $H<G$, which is essentially given by Drakokhrust. See \cite{Drakokhrust1989} and \cite[\S 5]{Macedo2022}. Here, we give a purely group-theoretical paraphrase of their theory. For a finite group $G$ and its subgroup $H$, put
\begin{equation*}
\Phi^{G}(H):=\langle [h,g]\in G \mid g\in G,h\in H\cap g^{-1}Hg \rangle. 
\end{equation*}
It is called the \emph{focal subgroup} of $H$ in $G$. 

\begin{prop}\label{prop:sosh}
Let $G$ be a finite group, and $H$ its subgroup. Assume that $\Sha_{\omega}^{2}(G,J_{G/H})$ is contained in the kernel of the homomorphism $\delta^{2}\colon H^{2}(G,J_{G/H})\rightarrow H^{3}(G,\Z)$ induced by the canonical exact sequence of $G$-lattices
\begin{equation*}
0\rightarrow \Z \xrightarrow{\varepsilon_{G/H}^{\circ}} \Ind_{H}^{G}\Z \rightarrow J_{G/H} \rightarrow 1. 
\end{equation*}
Then, there is an isomorphism
\begin{equation*}
\Sha_{\omega}^{2}(G,J_{G/H})\cong \left((H\cap G^{\der})/\Phi^{G}(H)\right)^{\vee}. 
\end{equation*}
\end{prop}

We need the following to prove Proposition \ref{prop:sosh}. It will be also used in the next section. 

\begin{prop}\label{prop:upcj}
Let $G$ be a finite group, $H$ its subgroup, and $\cD$ an admissible set of subgroups in $G$. Then, we have
\begin{equation*}
\Sha_{\cD}^{2}(G,J_{G/H})=\Sha_{\cD_{(H)}}^{2}(G,J_{G/H}),
\end{equation*}
where $\cD_{(H)}:=\cC_{H}\cup (\cD\setminus \cC_{G})$. In particular, one has an equality
\begin{equation*}
\Sha_{\omega}^{2}(G,J_{G/H})=\Sha_{\cC_{H}}^{2}(G,J_{G/H}). 
\end{equation*}
\end{prop}

\begin{proof}
This is a special case of \cite[Theorem 5.10]{Oki2025a}. 
\end{proof}

\begin{proof}[Proof of Proposition \ref{prop:sosh}]
Consider a commutative diagram
\begin{equation*}
\xymatrix@C=30pt{
H^{2}(G,\Z)\ar[r]^{\varepsilon_{G/H,*}^{\circ}\hspace{15pt}}\ar[d]^{\cong}& H^{2}(G,\Ind_{H}^{G}\Z)\ar[r]^{\hspace{5pt}j_{G/H,*}}\ar[d]^{\cong}& H^{2}(G,J_{G/H})\ar[r]& 0\\
G^{\vee}\ar[r]^{f\mapsto f\mid_{H}}&H^{\vee},&& 
}
\end{equation*}
where the horizontal sequences are exact. Then, the right vertical map induces an isomorphism
\begin{equation}\label{eq:imep}
\Ima(\varepsilon_{G/H,*}^{\circ})\cong (H/(H\cap G^{\der}))^{\vee}. 
\end{equation}
Now, let $\Sel_{G/H}^{2}:=j_{G/H,*}^{-1}(\Sha_{\omega}^{2}(G,J_{G/H}))$. By \eqref{eq:imep}, one has an exact sequence
\begin{equation}\label{eq:slsh}
0\rightarrow (H/(H\cap G^{\der}))^{\vee} \rightarrow \Sel_{G/H}^{2}\xrightarrow{j_{G/H,*}} \Sha_{\omega}^{2}(G,J_{G/H})\rightarrow 0. 
\end{equation}
Therefore, it suffices to prove that $H^{2}(G,\Ind_{H}^{G}\Z)\cong H^{\vee}$ induces an isomorphism
\begin{equation}\label{eq:slis}
\Sel_{G/H}^{2}\cong (H/\Phi^{G}(H))^{\vee}. 
\end{equation}
Recall Proposition \ref{prop:upcj} that one has
\begin{equation*}
\Sha_{\omega}^{2}(G,J_{G/H})=\Ker\left(H^{2}(G,J_{\widetilde{G}/H})\xrightarrow{(\Res_{G/D})_{D\in \cC_{H}}}\bigoplus_{D\in \cC_{H}}H^{2}(D,J_{G/H})\right). 
\end{equation*}
Using this, we obtain an equality
\begin{equation*}
\Sel_{G/H}^{2}=(\Res_{G/D})_{D\in \cC_{H}}^{-1}\left(\Ima(\varepsilon_{G/H,*,\loc}^{\circ})\right), 
\end{equation*}
where $(\Res_{G/D})_{D\in \cC_{H}}$ and $\varepsilon_{G/H,*,\loc}^{\circ}$ are as in the following commutative diagram: 
\begin{equation*}
\xymatrix@C=115pt{
H^{2}(G,\Z)\ar[r]^{\varepsilon_{G/H,*}^{\circ}}\ar[d]& 
H^{2}(G,\Ind_{H}^{G},\Z)\ar[d]^{(\Res_{G/D})_{D\in \cC_{G}}}\\
\bigoplus_{D\in \cC_{H}}H^{2}(D,\Z)\ar[r]^{\varepsilon_{G/H,*,\loc}^{\circ}=(\varepsilon_{G/H,*,D})_{D\in \cC_{H}}\hspace{15pt}}&\bigoplus_{D\in \cC_{H}}H^{2}(D,\Ind_{H}^{G}\Z). 
}
\end{equation*}
Combining these facts with Lemma \ref{lem:rsch}, we obtain an isomorphism
\begin{equation}\label{eq:slch}
\Sel_{G/H}^{2}\cong 
\{f\in H^{\vee}\mid f\circ \Ad(g^{-1})\mid_{D\cap gHg^{-1}}=f\!\mid_{D\cap gHg^{-1}}\text{ for any }D\in \cC_{H}\text{ and }g\in G\}. 
\end{equation}
Therefore, \eqref{eq:slis} is reduced to prove the coincidence between the right-hand side of \eqref{eq:slch} and $(H/\Phi^{G}(H))^{\vee}$. 

First, assume that $f\in H^{\vee}$ is contained in the right-hand side of \eqref{eq:slch}. Take $g\in G$ and $h\in H\cap g^{-1}Hg$. Put $D:=\langle h\rangle$. Then, we have $D<H\cap g^{-1}Hg$ and $f\!\mid_{D\cap g^{-1}Hg}=(f\circ \Ad(g))\!\mid_{D\cap g^{-1}Hg}$. This implies $f(h)=f(g^{-1}hg)$, and hence $f([h,g])=0$. In particular, we have $f(\Phi^{G}(H))=0$, that is, $f\in (H/\Phi^{G}(H))^{\vee}$. Next, for the reverse inclusion, take $f\in (H/\Phi^{G}(H))^{\vee}$. Take $D=\langle d \rangle \in \cC_{H}$ with $d\in H$ and $g\in G$. Assume $D\cap gHg^{-1}=\langle d^{m}\rangle$ for some positive integer $m$. By assumption, we have $f([d^{m},g^{-1}])=0$, and hence $(f\circ \Ad(g^{-1}))\!\mid_{D\cap gHg^{-1}}=f\!\mid_{D\cap gHg^{-1}}$. Hence, $f$ is contained in the right-hand side of \eqref{eq:slch}. This completes the proof. 
\end{proof}

Recall \cite[Proposition 2.13]{Beyl1982} and \cite[Definition 5.1]{Macedo2022} that a \emph{generalized representation group} of a finite group $G$ is a central extension
\begin{equation*}
1\rightarrow \widetilde{Z}\rightarrow \widetilde{G}\rightarrow G\rightarrow 1
\end{equation*}
such that the transgression map $H^{1}(\widetilde{Z},\Q/\Z)\rightarrow H^{2}(G,\Q/\Z)$ is surjective. 

\begin{cor}[{cf.~\cite[Theorem 5.3]{Macedo2022}}]\label{cor:drmn}
Let $G$ be a finite group, and $H$ its subgroup. Consider a generalized representation group
\begin{equation*}
1\rightarrow \widetilde{Z}\rightarrow \widetilde{G}\xrightarrow{\pi} G\rightarrow 1
\end{equation*}
of $G$, and put $\widetilde{H}:=\pi^{-1}(H)$. Then, there is an isomorphism
\begin{equation*}
\Sha_{\omega}^{2}(G,J_{G/H})\cong \left((\widetilde{H}\cap \widetilde{G}^{\der})/\Phi^{\widetilde{G}}(\widetilde{H})\right)^{\vee}.   
\end{equation*}
\end{cor}

\begin{proof}
By Proposition \ref{prop:ifts}, one has a commutative diagram
\begin{equation*}
\xymatrix{
\Sha_{\omega}^{2}(G,J_{G/H})\ar[r]^{\hspace{10pt}\delta_{G/H}^{2}}\ar[d]_{\cong}^{\Inf_{\widetilde{G}/G}}&H^{3}(G,\Z)\ar[d]^{\Inf_{\widetilde{G}/G}}\\
\Sha_{\omega}^{2}(\widetilde{G},J_{\widetilde{G}/\widetilde{H}})\ar[r]^{\hspace{10pt}\delta_{\widetilde{G}/\widetilde{H}}^{2}}&H^{3}(\widetilde{G},\Z),
}
\end{equation*}
where the vertical homomorphisms are the inflation maps. Since $H^{1}(\widetilde{Z},\Q/\Z)\rightarrow H^{2}(G,\Q/\Z)$ is surjective, we obtain that the right vertical map in the above diagram is zero. Hence, the assertion follows from Proposition \ref{prop:sosh}. 
\end{proof}

\subsection{The case of prime power degree}

\begin{prop}\label{prop:rdpp}
Let $p$ be a prime number, and $G$ a transitive group of degree a power of $p$, and $H$ a one point stabilizer in $G$. Then, we have 
\begin{equation*}
\Sha_{\cD}^{2}(G,J_{G/H})=\Sha_{\cD_{p}}^{2}(G,J_{G/H})[p^{\infty}],
\end{equation*}
where $\cD_{p}$ is the set of $p$-Sylow subgroups of all elements of $\cD$. 
\end{prop}

\begin{proof}
We have $\Sha_{\cD}^{2}(G,J_{G/H})\subset H^{2}(G,J_{G/H})[p^{\infty}]$, which is a consequence of Proposition \ref{prop:andg}. Hence, the assertion follows from Proposition \ref{prop:shrs}. 
\end{proof}

\begin{prop}\label{prop:shpp}
Let $p$, $G$, $H$, and $\cD$ be as in Proposition \ref{prop:rdpp}. Consider a subgroup $G_{1}$ of $G$ that contains a $p$-Sylow subgroup of $G$. 
\begin{enumerate}
\item There is an isomorphism of $G_{1}$-lattices $J_{G/H}\cong J_{G_{1}/(G_{1}\cap H)}$. 
\item The homomorphism
\begin{equation*}
\Res_{G/G_{1}}\colon \Sha_{\cD}^{2}(G,J_{G/H})\rightarrow \Sha_{\cD_{\cap G_{1}}}^{2}(G_{1},J_{G_{1}/(G_{1}\cap H)})
\end{equation*}
is injective. 
\end{enumerate}
\end{prop}

\begin{proof}
(i): Let $P$ be a $p$-Sylow subgroup of $G$ that is contained in $G_{1}$. By Lemma \ref{lem:sytr}, $P$ is a transitive group with $P\cap H$ a one point stabilizer. Hence, the double coset $G_{1}\backslash G/H$ is singuleton. In particular, $\{1\}$ is a complete representative of $G_{1}\backslash G/H$ in $G$. Therefore, the assertion follows from \cite[Proposition 3.20]{Hasegawa2025}. 

(ii): By Proposition \ref{prop:andg}, we have $\Sha_{\cD}^{2}(G,J_{G/H})=\Sha_{\cD}^{2}(G,J_{G/H})[p^{\infty}]$. Combining this fact with Proposition \ref{prop:shrs}, we obtain the desired assertion. 
\end{proof}

If $(G:H)$ is a prime number, then the structure of $\Sha_{\cD}^{2}(G,J_{G/H})$ is completely determined. 

\begin{prop}[{\cite[Proposition 5.15]{Oki2025a}; cf.~\cite[Lemma 4]{Bartels1981a}}]\label{prop:btls}
Let $G$ be a finite group, and $H$ a subgroup of $G$ of prime index. Then, 
\begin{equation*}
\Sha_{\omega}^{2}(G,J_{G/H})=0. 
\end{equation*}
\end{prop}

\section{Computation of second cohomology groups: The case of $p$-groups}\label{sect:csha}

\begin{lem}[{\cite[Lemma 6.1]{Oki2025a}}]\label{lem:jgcg}
Let $G$ be a finite group, and $H<H'$ subgroups of $G$. Then, there is an exact sequence
\begin{equation*}
\xymatrix@C=35pt{
0\ar[r] & \Ind_{H'}^{G}\Z \ar[d] \ar[r]^{\Ind_{H'}^{G}\varepsilon_{H'}^{\circ}}& \Ind_{H}^{G}\Z \ar[d] \ar[r]^{\Ind_{H'}^{G}j_{E}\hspace{10pt}}& \Ind_{H'}^{G}J_{H'/H} \ar@{=}[d] \ar[r]& 0\\
0\ar[r] & J_{G/H'} \ar[r]& J_{G/H} \ar[r]^{j\hspace{15pt}}& \Ind_{H'}^{G}J_{H'/H} \ar[r]& 0, }
\end{equation*}
where the horizontal sequences are exact. 
\end{lem}

Keep the notation in Lemma \ref{lem:jgcg}. Let $D$ be a subgroup of $G$. Then, we obtain an exact sequence of finite abelian groups
\begin{equation}\label{eq:long}
H^{1}(D,\Ind_{H'}^{G}J_{H'/H})\xrightarrow{\delta_{D}} H^{2}(D,J_{G/H'})\xrightarrow{j_{*}} H^{2}(D,J_{G/H})\rightarrow H^{2}(D,\Ind_{H'}^{G}J_{H'/H}). 
\end{equation}

\begin{dfn}[{cf.~\cite[Definition 6.5]{Oki2025a}}]\label{dfn:sldf}
Let $G$ be a finite group, and $H<H'$ subgroups of $G$. For an admissible set $\cD$ of subgroups in $G$, put
\begin{equation*}
\Sel_{H,\cD}^{2}(G,J_{G/H'}):=j_{*}^{-1}(\Sha_{\cD}^{2}(G,J_{G/H}))<H^{2}(G,J_{G/H'}). 
\end{equation*}
\end{dfn}

\begin{lem}\label{lem:ssex}
Retain the notation in Definition \ref{dfn:sldf}. 
\begin{enumerate}
\item One has an equality
\begin{equation*}
\Sel_{H,\cD}^{2}(G,J_{G/H'})=\{f\in H^{2}(G,J_{G/H'})\mid \Res_{G/D}(f)\in \Ima(\delta_{D})\text{ for any }D\in \cD_{(H)}\}. 
\end{equation*}
Here, $\cH_{(H)}$ is as in Proposition \ref{prop:upcj}, and $\delta_{D}$ is as in \eqref{eq:long}. 
\item We further assume $\Sha_{\omega}^{2}(H',J_{H'/H})=0$. Then, one has an exact sequence
\begin{equation*}
H^{1}(G,J_{G/H})\rightarrow H^{1}(H',J_{H'/H})\rightarrow \Sel_{H,\cD}^{2}(G,J_{G/H'})\xrightarrow{j_{*}} \Sha_{\cD}^{2}(G,J_{G/H})\rightarrow 0. 
\end{equation*}
\end{enumerate}
\end{lem}

\begin{dfn}\label{dfn:ssdf}
Let $G$, $H$ and $H'$ be as in Definition \ref{dfn:sldf}. Then, we define $\sS_{H,\cD}^{2}(G,\Ind_{H'}^{G}\Z)$ as the preimage of $\Sel_{H,\cD}^{2}(G,J_{G/H'})$ under the homomorphism
\begin{equation*}
j_{G/H',*}^{2}\colon H^{2}(G,\Ind_{H'}^{G}\Z)\rightarrow H^{2}(G,J_{G/H'}). 
\end{equation*}
If $\cD=\cC_{G}$, we simply denote $\sS_{H,\cD}^{2}(G,\Ind_{H'}^{G}\Z)$ by $\sS_{H}^{2}(G,\Ind_{H'}^{G}\Z)$. 
\end{dfn}

\begin{lem}\label{lem:sssl}
Retain the notations in Definitions \ref{dfn:sldf} and \ref{dfn:ssdf}. We further assume
\begin{itemize}
\item[(a)] $\Sha_{\omega}^{2}(H',J_{H'/H})=0$; and
\item[(b)] $\Res_{G/H'}\colon H^{3}(G,\Z) \rightarrow H^{3}(H',\Z)$ is injective. 
\end{itemize}
Then, there is an exact sequence
\begin{equation*}
0\rightarrow (H'/(H'\cap G^{\der}))^{\vee} \rightarrow \sS_{H,\cD}^{2}(G,\Ind_{H'}^{G}\Z) \rightarrow \Sel_{H,\cD}^{2}(G,J_{G/H'}) \rightarrow 0. 
\end{equation*}
\end{lem}

\begin{proof}
By Lemma \ref{lem:shsh}, we obtain an injection $\Sha_{\cD}^{2}(G,\Ind_{H'}^{G}J_{H'/H})\hookrightarrow \Sha_{\omega}^{2}(H',J_{H'/H})$. Hence, 
\begin{equation*}
\Sha_{\cD}^{2}(G,\Ind_{H'}^{G}J_{H'/H})=0
\end{equation*}
by (a). This produces a commutative diagram
\begin{equation*}
\xymatrix{
\Sel_{H,\cD}^{2}(G,J_{G/H'}) \ar[r]\ar[d]^{\delta_{G/H',*}^{2}}&
\Sha_{\cD}^{2}(G,J_{G/H}) \ar[r]\ar[d]^{\delta_{G/H,*}^{2}}&
\Sha_{\cD}^{2}(G,\Ind_{H'}^{G}J_{H'/H})=0\ar[d]\\
H^{3}(G,\Z) \ar@{=}[r]&H^{3}(G,\Z) \ar[r]^{\Res_{G/E}}&H^{3}(H',\Z). 
}
\end{equation*}
On the other hand, (b) implies that the leftmost vertical map is zero. Therefore, the exact sequence
\begin{equation*}
H^{2}(G,\Z)\xrightarrow{\varepsilon_{G/H',*}^{\circ,2}} H^{2}(G,\Ind_{H'}^{G}\Z)\xrightarrow{j_{G/H',*}^{2}} H^{2}(G,J_{G/H'}) \rightarrow H^{3}(G,\Z)
\end{equation*}
induces the desired exact sequence. 
\end{proof}

We focus on a specific case, which will be needed in this section. 

\begin{lem}\label{lem:pnor}
Let $p$ be a prime number, and $G$ a finite group that admits its normal $p$-subgroup $N$. Consider a subgroup $H$ of $N$ with $\Sha_{\omega}^{2}(N,J_{N/H})=0$. Then, there is an exact sequence
\begin{equation*}
0\rightarrow ((N \cap HG^{\der})/HN^{\der})^{\vee}\rightarrow \Sel_{H,\cD}^{2}(G,J_{G/N}) \xrightarrow{j_{*}} \Sha_{\cD}^{2}(G,J_{G/H}) \rightarrow 0. 
\end{equation*}
\end{lem}

\begin{proof}
By Lemma \ref{lem:ssex}, it suffices to give an isomorphism
\begin{equation*}
\Coker(H^{1}(G,J_{G/N})\rightarrow H^{1}(N,J_{N/H}))\cong ((N\cap HG^{\der})/H)^{\vee}. 
\end{equation*}
This follows from the commutative diagram
\begin{equation*}
\xymatrix{
H^{1}(G,J_{G/N})\ar[r]\ar[d]^{\cong}&
H^{1}(N,J_{N/H}))\ar[d]^{\cong}\\
(G/G^{\der})^{\vee}\ar[r]^{f\mapsto f\mid_{E}}&(N/H)^{\vee}
}
\end{equation*}
and the isomorphism $(N/(N\cap G^{\der}))^{\vee}/(N/HN^{\der})^{\vee}\cong ((N\cap G^{\der})/HN^{\der})^{\vee}$. 
\end{proof}

\begin{lem}\label{lem:rsht}
Retain the notation and assumptions in Lemma \ref{lem:pnor}. 
Let $D$ be a subgroup of $N$. 
\begin{enumerate}
\item There is a commutative diagram
\begin{equation*}
\xymatrix{
H^{1}(D,\Ind_{N}^{G}J_{N/H}) \ar[d]^{\cong}\ar[r]^{\hspace{15pt}\delta_{D}} & H^{2}(D,\Ind_{N}^{G}\Z) \ar[d]^{\cong}\\
\bigoplus_{g\in R_{G}(D,N)}C_{D,g} \ar@{^{(}->}[r]& \bigoplus_{g\in R_{G}(D,N)}(D \cap N)^{\vee}. }
\end{equation*}
Here we use the notation as follows: 
\begin{itemize}
\item $R_{G}(D,N)$ is a complete representative of $D\backslash G/N$ in $G$; 
\item for each $g\in R_{G}(D,N)$, $C_{D,g}$ is the kernel of the homomorphism
\begin{equation*}
D^{\vee}\rightarrow \bigoplus_{h \in R_{N}(D,H,g)}(D\cap \Ad(gh)H)^{\vee};f\mapsto (f\!\mid_{D\cap \Ad(gh)H})_{h}, 
\end{equation*}
where $R_{N}(D,H,g)$ is a complete representative of $(g^{-1}Dg \cap N)\backslash N/H$ in $N$. 
\end{itemize}
\item There is a commutative diagram
\begin{equation*}
\xymatrix@C=40pt{
H^{2}(G,\Ind_{N}^{G}\Z)[p^{\infty}] \ar[d]^{\cong}\ar[r]^{\hspace{5pt}\Res_{G/D}} & H^{2}(D,\Ind_{N}^{G}\Z) \ar[d]^{\cong} \\
N^{\vee} \ar[r]^{f\mapsto f\circ (\Ad(g^{-1})\mid_{D})_{g}\hspace{45pt}}& \bigoplus_{g\in R_{G}(D,N)}(D \cap N)^{\vee}. }
\end{equation*}
\end{enumerate}
\end{lem}

\begin{proof}
We follow the proof of \cite[Lemma 6.13]{Oki2025a}. 

(i): By Proposition \ref{prop:mcky}, we have isomorphisms of $D$-lattices
\begin{gather*}
\Ind_{N}^{G}\Z \cong \bigoplus_{g\in R_{G}(D,N)}\Z, \\
\Ind_{N}^{G}J_{N/H}\cong \bigoplus_{g\in R_{G}(D,N)}(J_{N/H})^{g}\cong \bigoplus_{g\in R_{G}(D,N)}J_{N/gHg^{-1}}. 
\end{gather*}
Combining these with Proposition \ref{prop:mcky}, we obtain a commutative diagram
\begin{equation*}
\xymatrix@C=50pt{
H^{1}(D,\Ind_{N}^{G}J_{N/H})\ar[r]^{\Ind_{N}^{G}\delta_{N/H}^{1}} \ar[d]^{\cong} & H^{2}(D,\Ind_{N}^{G}\Z)\ar[d]^{\cong}\\
\bigoplus_{g\in R_{G}(D,N)}H^{1}(D,J_{N/gHg^{-1}})\ar[r]^{\hspace{10pt}(\delta_{N/gHg^{-1}}^{1})_{g}} & \bigoplus_{g\in R_{G}(D,N)}H^{2}(D,\Z). }
\end{equation*}
Here, the bottom homomorphism is the direct sum of the connecting homomorphism induced by the canonical exact sequence of $D$-lattices
\begin{equation*}
0 \rightarrow \Z \rightarrow \Ind_{gHg^{-1}}^{N}\Z \rightarrow J_{N/gHg^{-1}} \rightarrow 0. 
\end{equation*}
Hence, it suffices to prove the commutativity of the following diagram for any $g\in R_{G}(D,N)$: 
\begin{equation}\label{eq:dcf1}
\begin{gathered}
\xymatrix@C=45pt{
H^{1}(D,J_{N/gHg^{-1}})\ar[r]^{\hspace{15pt}\delta_{N/gHg^{-1}}^{1}} \ar[d]^{\cong} & H^{2}(D,\Z)\ar[d]^{\cong}\\
C_{D,g}\ar@{^{(}->}[r] & D^{\vee}. }
\end{gathered}
\end{equation}
Fix $g\in R_{G}(D,N)$. Since every $h\in R_{N}(D,H,g)$ satisfies
\begin{equation*}
\Ad(ghg^{-1})(gHg^{-1})=\Ad(gh)H<N,
\end{equation*}
Then Proposition \ref{prop:mcky} implies an isomorphism of $D$-lattices
\begin{equation*}
\Ind_{gHg^{-1}}^{N}\Z \cong \bigoplus_{h\in R_{N}(D,H,g)}\Ind_{D\cap \Ad(gh)H}^{D}\Z. 
\end{equation*}
Hence, Proposition \ref{lem:htch} induces the desired commutative diagram \eqref{eq:dcf1}. 

(ii): This follows from Lemma \ref{lem:rsch}. 
\end{proof}

\begin{lem}\label{lem:sshe}
Retain the notations and assumptions in Lemma \ref{lem:pnor}. We further assume that $N$ is abelian. 
\begin{enumerate}
\item The isomorphism $H^{2}(G,\Ind_{N}^{G}\Z)\cong N^{\vee}$ induces
\begin{equation}\label{eq:ssch}
\sS_{H,\cD}^{2}(G,\Ind_{N}^{G}\Z)\cong \{f\in N^{\vee}\mid 
(f\circ \Ad(g^{-1})\!\mid_{D\cap N})_{g}\in \Delta_{D}+\cI_{D}\text{ for all }D\in \cD_{(H)}\}. 
\end{equation}
Here, $\Delta_{D}$ is the diagonal subset in $\bigoplus_{R_{G}(D,N)}((D\cap N)/(D^{\der}\cap N))^{\vee}$, and
\begin{equation*}
\cI_{D}:=\bigoplus_{g\in R_{G}(D,N)}((D\cap N)/(D\cap gHg^{-1}))^{\vee}. 
\end{equation*}
\item The subgroup $(N/(H\cap G^{\der}))^{\vee}$ of $N^{\vee}$ is contained in the right-hand side of \eqref{eq:ssch}. 
\item Let $f\in N^{\vee}$. Then, the following are equivalent:
\begin{itemize}
\item[(a)] $(f\circ \Ad(g^{-1})\!\mid_{D})_{g}\in \Delta_{D}$ for any $D\in \cC_{N^{G}(H)}$; 
\item[(b)] $f=0$ on $[N^{G}(H),G]$. 
\end{itemize}
In particular, $\sS_{H,\cD}^{2}(G,\Ind_{N}^{G}\Z)$ is contained in $(N/[N^{G}(H),G])^{\vee}$. 
\end{enumerate}
\end{lem}

\begin{proof}
(i): We follow the proof of \cite[Proposition 6.14]{Oki2025a}. Recall that one has a commutative diagram as follows for each $D\in \cD_{(H)}$: 
\begin{equation*}
\xymatrix@C=50pt{
&H^{2}(G,\Ind_{N}^{G}\Z) \ar[r]^{j_{G/N,*}^{2}}\ar[d]^{\Res_{G/D}}&H^{2}(G,J_{G/N}) \ar[r]\ar[d]^{\Res_{G/D}}&0\\
H^{2}(D,\Z)\ar[r]^{\hspace{5pt}\varepsilon_{D}:=\varepsilon_{G/N,*}^{2}\hspace{20pt}}&H^{2}(D,\Ind_{N}^{G}\Z)\ar[r]^{j_{G/N,*}^{2}}&H^{2}(D,J_{G/N})\ar[r]&0. 
}
\end{equation*}
Moreover, by Lemma \ref{lem:jgcg}, the homomorphism $\delta_{D}$ in Lemma \ref{lem:rsht} (i) is factored as
\begin{equation*}
H^{1}(D,\Ind_{N}^{G}J_{N/H})\xrightarrow{\Ind_{N}^{G}\delta_{N/H}^{1}} H^{2}(D,\Ind_{N}^{G}\Z)\xrightarrow{j_{G/N,*}^{2}} H^{2}(D,J_{G/N}). 
\end{equation*}
Consequently, we obtain an equality
\begin{equation*}
\sS_{H}^{2}(G,\Ind_{N}^{G}\Z)=\{f\in H^{2}(G,\Ind_{N}^{G}\Z)\mid \Res_{G/D}(f)\in \Ima(\delta_{D})+\Ima(\varepsilon_{D})\text{ for any }D\in \cD_{(H)}\}. 
\end{equation*}
Now, fix $D\in \cD_{(H)}$. Then we have
\begin{equation*}
\Ima(\delta_{D})=\bigoplus_{g\in R(D,N)}C_{D,g}=\cI_{D}. 
\end{equation*}
Furthermore, since $N$ is abelian, we have $D\cap \Ad(gh)H=D\cap gHg^{-1}$ for any $D\in \cD_{(H)}$, $g\in G$ and $h\in H$. This implies an equality
\begin{equation*}
\bigoplus_{g\in R_{G}(D,N)}C_{D,g}=\cI_{D}. 
\end{equation*}
On the other hand, by Lemma \ref{lem:rsch}, $H^{2}(D,\Ind_{N}^{G}\Z)\cong C_{D,N}$ in Lemma \ref{lem:rsht} (i) induces an isomorphism
\begin{equation*}
\Ima(\varepsilon_{D})\cong \Delta_{D}. 
\end{equation*}
Hence, applying Lemma \ref{lem:rsht} (ii), we obtain the desired isomorphism. 

(ii): By Lemma \ref{lem:jgcg}, the homomorphism $\delta_{G}$ in Lemma \ref{lem:rsht} (i) is factored as
\begin{equation*}
H^{1}(G,\Ind_{N}^{G}J_{N/H})\xrightarrow{\Ind_{N}^{G}\delta_{N/H}^{1}} H^{2}(G,\Ind_{N}^{G}\Z)\xrightarrow{j_{G/N,*}^{2}} H^{2}(G,J_{G/N}). 
\end{equation*}
On the other hand, one has a commutative diagram
\begin{equation*}
\xymatrix@C=45pt{
H^{1}(G,\Ind_{N}^{G}J_{N/H})\ar[r]^{\hspace{8pt}\Ind_{N}^{G}\delta_{N/H}^{1}}\ar[d]^{\cong}& H^{2}(G,\Ind_{N}^{G}\Z)\ar[d]^{\cong}\\
(N/H)^{\vee}\ar@{^{(}->}[r]& N^{\vee}. 
}
\end{equation*}
Moreover, $\varepsilon_{G/N}^{\circ}$ induces a commutative diagram
\begin{equation}\label{eq:icss}
\begin{gathered}
\xymatrix@C=35pt{
H^{2}(G,\Z)\ar[r]^{\varepsilon_{G/N,*}^{2}\hspace{15pt}}\ar[d]^{\cong}& H^{2}(G,\Ind_{N}^{G}\Z)\ar[d]^{\cong}\\
G^{\vee}\ar[r]^{f \mapsto f\mid_{N}}& N^{\vee}. 
}
\end{gathered}
\end{equation}
The image of the bottom horizontal map in \eqref{eq:icss} coincides with $(N/(N\cap G^{\der}))^{\vee}$. Hence, the assertion holds since $H<N$. 

(iii): Pick $f\in H^{2}(G,\Ind_{N}^{G}\Z)$. By definition, we have $\cI_{D}=0$ for any $D\in \cC_{N^{G}(H)}$. Hence, $f$ satisfies (a) if and only if $f(n)=f(gng^{-1})$ for all $n\in N^{G}(H)$ and $g\in R_{G}(D,N)\in \Delta_{D}$. This is equivalent to $f([N^{G}(H),G])=0$, since $R_{G}(D,N)$ also represents the quotient $G/N$. Therefore, the assertion holds. 
\end{proof}

For use of Lemma \ref{lem:sssl}, we need a study of restriction maps for the third cohomology groups in coefficient $\Z$ for certain finite groups. The following will be helpful in investigating them. 

\begin{prop}[{\cite[Theorem 2 (ii)]{Tahara1972}}]\label{prop:thrt}
Let $G=N\rtimes H$, where $N$ and $H$ are finite groups. For a trivial $G$-module $A$, there is an exact sequence of abelian groups
\begin{equation*}
0\rightarrow H^{1}(H,\Hom(N,A))\rightarrow H_{(H)}^{2}(G,A)\xrightarrow{\Res_{G/N}} H^{2}(N,A)^{H}\rightarrow H^{2}(H,\Hom(N,A)). 
\end{equation*}
Here, $H_{(H)}^{2}(G,A)$ is the kernel of the restriction map $\Res_{G/H}\colon H^{2}(G,A)\rightarrow H^{2}(H,A)$. 
\end{prop}

\begin{prop}\label{prop:bgtr}
Let $p>2$ be an odd prime number, and $n\in \{1,\ldots,p\}$. Then the restrcition map 
\begin{equation*}
\Res_{P_{n}/N_{n}}\colon H^{3}(P_{n},\Z)\rightarrow H^{3}(N_{n},\Z)
\end{equation*}
is injective. 
\end{prop}

\begin{proof}
It suffices to prove the injectivity of the restriction map
\begin{equation*}
H^{2}(P_{n},\Q/\Z)\rightarrow H^{2}(N_{n},\Q/\Z). 
\end{equation*}

\emph{Case 1.~$n=1$. }By definition we have $P_{1}\cong C_{p^{2}}$. This implies $H^{2}(P_{1},\Q/\Z)=0$. 

\emph{Case 2.~$n=2$. }By Lemma \ref{lem:pnpr} (ii), $P_{2}$ is extraspecial of exponent $p^{2}$. Hence, we have $H^{2}(P_{2},\Q/\Z)=0$ by \cite[Theorem 3.3.6 (ii)]{Karpilovsky1987}. 

\emph{Case 3.~$n=p$. }By definition, we have $P_{p}=N_{p}\rtimes \langle \rho_{2} \rangle$. Moreover, if we regard $N_{p}$ as an $\Fp[\rho_{2}]$-module, there is an isomorphism $N_{p}\cong \Fp[\rho_{2}]$. This implies $H^{2}(\langle \rho_{2}\rangle,N_{p}^{\vee})=0$. On the other hand, we have $H_{\langle \rho_{2} \rangle}^{2}(P_{p},\Q/\Z)=H^{2}(P_{p},\Q/\Z)$ since $H^{2}(\langle \rho_{2}\rangle,\Q/\Z)=0$. Therefore, the assertion follows from Proposition \ref{prop:thrt}. 

\emph{Case 4.~Remaining cases. }This case happens only when $p\geq 5$ and $n\in \{3,\ldots,p-1\}$. Then, Lemma \ref{lem:pnpr} (iii) gives an exact sequence
\begin{equation*}
0\rightarrow B_{0}(P_{n})\rightarrow H^{2}(P_{n},\Q/\Z)\rightarrow H^{2}(N_{n},\Q/\Z). 
\end{equation*}
Here, $B_{0}(P_{n})$ is the Bogomolov multipler of $P_{n}$, that is, 
\begin{equation*}
B_{0}(P_{n}):=\Ker\left((\Res_{G/A})_{A}\colon H^{2}(P_{n},\Q/\Z)\rightarrow \bigoplus_{A<G,A^{\der}=\{1\}}H^{2}(A,\Q/\Z)\right). 
\end{equation*}
Recall that $P_{n-1}$ is a maximal class, which follow from Lemma \ref{lem:pnpr} (iii). On the other hand, consider the descending sequence
\begin{equation*}
P_{n}>N_{n}>N_{n-1}>\cdots >N_{1}>N_{0}=\{1\},
\end{equation*}
which is a chief series. Then, it satisfies $P_{n-1}^{\der}=N_{n-2}$ and $[P_{i}:G]=P_{i-1}$ for any $i\in \{1,\ldots,n-1\}$, which follows from Lemma \ref{lem:pnpr} (iii). Moreover, we have $N_{n-1}=Z_{P_{n-1}}(N_{n-2}/N_{n-4})$. In addition, since $N_{n-1}$ is abelian, one has $[N_{n-1},N_{n-1}]=[N_{n-1},N_{2}]=\{1\}$. Hence, we obtain $B_{0}(P_{n-1})=0$ by \cite[Theorem 1.1]{FernandezAlcober2020}. This completes the proof. 
\end{proof}

\begin{prop}\label{prop:g1vn}
Let $p$ be a prime number, and $n\in \{1,\ldots,p\}$. Then, we have
\begin{equation*}
\sS_{H_{n}}^{2}(P_{n},\Ind_{N_{n}}^{P_{n}}\Z)\cong (N_{n}/(H_{n}\cap N_{n-1}))^{\vee}. 
\end{equation*}
\end{prop}

\begin{proof}
Recall Lemma \ref{lem:enel} (iii) that $P_{n}^{\der}=N_{n-1}$ holds. Hence, Lemma \ref{lem:sshe} (ii) implies that $\sS_{{H}_{n}}^{2}(P_{n},\Ind_{N_{n}}^{P_{n}}\Z)$ contains $(N_{n}/(H_{n}\cap N_{n-1}))^{\vee}$. On the other hand, one has the following by definition: 
\begin{equation*}
H_{n}\cap N_{n-1}=\langle [\eta_{n},\rho_{2}],\ldots,[\eta_{n},\rho_{2}^{n-2}]\rangle. 
\end{equation*}
Take any $f\in \sS_{H_{n}}^{2}(P_{n},\Ind_{N_{n}}^{P_{n}}\Z)$, and regard it as a character of $N_{n}$. Put $D:=\langle \eta_{n}\rangle$, then $\langle \rho_{2}\rangle$ is a complete representative of $D\backslash P_{n}/N_{n}$ in $P_{n}$. Moreover, for $i\in \{0,\ldots,p-1\}$, we have $\rho_{2}^{i}D\rho_{2}^{-i}\subset H_{n}$ if and only if $i\leq n-2$. Hence, we have
\begin{equation*}
\Delta_{D}+\cI_{D}=\left\{(\psi_{g})_{g}\in \bigoplus_{g\in \langle \rho_{2}\rangle}D^{\vee}\,\,\middle|\,\, \psi_{1}=\cdots=\psi_{\rho_{2}^{n-2}}\right\}, 
\end{equation*}
where $\Delta_{D}$ and $\cI_{D}$ are as in Lemma \ref{lem:sshe} (i). In particular, we obtain an equality $f(\eta_{n})=f(\rho_{2}^{i}\eta_{n}\rho_{2}^{-i})$, that is, 
\begin{equation*}
f([\eta_{n},\rho_{2}^{i}])=0
\end{equation*}
for any $i\in \{0,\ldots,n-2\}$. Hence, the proof is complete. 
\end{proof}

\begin{prop}\label{prop:ppsh}
Let $p>2$ be an odd prime number, and $n\in \{1,\ldots,p\}$. Set 
\begin{equation*}
\widetilde{H}_{n}:=\langle \eta_{p},\rho_{2}\eta_{p}\rho_{2}^{-1},\ldots,\rho_{2}^{n-2}\eta_{p}\rho_{2}^{-(n-2)}\rangle \cdot N_{p-n},
\end{equation*}
which is contained in $N_{p}$ and $P_{p}$. 
\begin{enumerate}
\item The subgroup $\widetilde{H}_{n}$ is the preimage of $H_{n}$ under the surjection $\pi_{n}\colon P_{p}\twoheadrightarrow P'_{n}$ in Lemma \ref{lem:pmpr} \emph{(ii)}. 
\item There is an isomorphism
\begin{equation*}
\sS_{\widetilde{H}_{n}}^{2}(P_{p},\Ind_{N_{p}}^{P_{p}}\Z)\cong 
\begin{cases}
(N_{p}/(\widetilde{H}_{n}\cap N_{p-1}))^{\vee}&\text{if }n\geq 3;\\
(N_{p}/N_{p-n-1})^{\vee}&\text{if }n\leq 2. 
\end{cases}
\end{equation*}
\end{enumerate}
\end{prop}

\begin{proof}
(i): By definition, we have $\widetilde{H}_{n}=\langle \eta_{p},\eta_{p-1},\ldots,\eta_{p-(n-2)}\rangle \cdot N_{p-n}$. 
Hence, the assertion follows from the definition of $\pi_{n}$, that is, $\pi_{n}(\eta_{n'})=\eta_{n'-n}$ for any $n'\in \{n+1,\ldots,p\}$. 

(ii): Recall Lemma \ref{lem:enel} (iii) that $P_{p}^{\der}=N_{p-1}$ holds. Hence, Lemma \ref{lem:sshe} (ii) implies that $\sS_{\widetilde{H}_{n}}^{2}(P_{p},\Ind_{N_{p}}^{P_{p}}\Z)$ contains $(N_{p}/(\widetilde{H}_{n}\cap N_{p-1}))^{\vee}$. In addition, we have
\begin{equation*}
\widetilde{H}_{n}\cap N_{p-1}=\langle [\eta_{p},\rho_{2}],\ldots,[\eta_{p},\rho_{2}^{n-2}]\rangle \cdot N_{p-n}
\end{equation*}
since $[N_{p},P_{p}]=N_{p-1}$. On the other hand, $\sS_{\widetilde{H}_{n}}^{2}(P_{p},\Ind_{N_{p}}^{P_{p}}\Z)$ is contained in $(N_{p}/N_{p-3})^{\vee}$, which follows from Lemma \ref{lem:sshe} (iii). In addition, note that $\langle \rho_{2}\rangle$ is a complete representative of $D\backslash P_{p}/N_{p}$ for any $D\in \cC_{\widetilde{H}_{n}}$. 

\vspace{6pt}
\emph{Case 1.~$n\geq 3$. }
Pick $f\in \sS_{\widetilde{H}_{n}}^{2}(P_{p},\Ind_{N_{p}}^{P_{p}}\Z)$. It suffices to prove $f([\eta_{p},\rho_{2}^{i}])=0$ for all $i\in \{1,\ldots,n-2\}$ and $f(\eta_{p-n})=0$. Set $D_{0}:=\langle \eta_{p}\rangle$. By direct computation, for $i\in \{0,\ldots,p-1\}$, we have $\rho_{2}^{i}D_{0}\rho_{2}^{-i}\subset \widetilde{H}_{n}$ if and only if $i\in \{1,\ldots,n-2\}$. Hence, we obtain an equality
\begin{equation*}
\Delta_{D_{0}}+\cI_{D_{0}}=\left\{(\psi_{g})_{g\in \langle \rho_{2}\rangle}\in \bigoplus_{\langle \rho_{2}\rangle}D_{0}^{\vee}\,\,\middle|\,\, \psi_{1}=\psi_{\rho_{2}}=\cdots=\psi_{\rho_{2}^{n-2}}\right\},
\end{equation*}
where $\Delta_{D_{0}}$ and $\cI_{D_{0}}$ is as in Lemma \ref{lem:sshe}. As a consequence, we have $f(\eta_{p})=f(\rho_{2}^{i}\eta_{p}\rho_{2}^{-i})$, that is,
\begin{equation}\label{eq:fdrt}
f([\eta_{p},\rho_{2}^{i}])=0,
\end{equation}
for every $i\in \{1,\ldots,n-2\}$. Now, recall the definition of $\eta_{p-n}$, and hence the following hold: 
\begin{equation}\label{eq:dpnd}
\eta_{p-n}=(\rho_{2}^{n-1}(\eta_{p}^{-n}\rho_{2}\eta_{p}\rho_{2}^{-1})\rho_{2}^{-(n-1)})\cdot \eta_{p}^{n-1}\cdot \prod_{i=1}^{n-2}[\eta_{p},\rho_{2}^{i}]^{a_{n,i}}. 
\end{equation}
Let $D_{1}:=\langle \eta_{p}^{-n}\rho_{2}\eta_{p}\rho_{2}^{-1} \rangle$, which satisfies $\rho_{2}^{n-1}D_{1}\rho_{2}^{-(n-1)}\subset \widetilde{H}_{n}$ by \eqref{eq:fdrt}. Hence, if an element $(\psi_{g})_{g\in \langle \rho_{2}\rangle}$ of $\bigoplus_{\langle \rho_{2}\rangle}D_{1}^{\vee}$ lies in $\Delta_{D_{1}}+\cI_{D_{1}}$, then $\psi_{1}=\psi_{\rho_{2}^{n-1}}$ holds. Consequently, one has 
\begin{equation*}
f(\rho_{2}^{n-1}(\eta_{p}^{-n}\rho_{2}\eta_{p}\rho_{2}^{-1})\rho_{2}^{-(n-1)})=f(\eta_{p}^{-n})+f(\rho_{2}\eta_{p}\rho_{2}^{-1})=(1-n)f(\eta_{p}),
\end{equation*}
where the second equality follows from \eqref{eq:fdrt}. Combining this with \eqref{eq:dpnd}, we obtain
\begin{equation*}
f(\eta_{p-n})=\sum_{i=1}^{n-2}a_{n,i}f([\eta_{p},\rho_{2}^{i}]). 
\end{equation*}
Now, the right-hand side of the above equality is zero by \eqref{eq:fdrt}, and therefore the proof is complete. 

\vspace{6pt}
\emph{Case 2.~$n=2$. }
Take $f\in H^{2}(P_{p},\Ind_{N_{p}}^{P_{p}}\Z)$. In this case, we have $\widetilde{H}_{n}\cap N_{p-1}=N_{p-n}$. Take $f\in \sS_{\widetilde{H}_{n}}^{2}(P_{p},\Ind_{N_{p}}^{P_{p}}\Z)$. By Lemma \ref{lem:sshe} (i), (iii), we have $f\in \sS_{\widetilde{H}_{2}}^{2}(P_{p},\Ind_{N_{p}}^{P_{p}}\Z)$ if and only if 
\begin{itemize}
\item[(1)] $f\in (N_{p}/N_{p-3})^{\vee}$; and 
\item[(2)] $(f\circ \Ad(g^{-1}))_{g\in \langle \rho_{2}\rangle}\in \Delta_{D}+\cI_{D}$ for any $D\in \cC_{\widetilde{H}_{2}}\setminus \cC_{N^{P_{p}}(\widetilde{H}_{2})}$. 
\end{itemize}
On the other hand, if $D\in \cC_{\widetilde{H}_{2}}\setminus \cC_{N^{P_{p}}(\widetilde{H}_{2})}$ and $i\in \{0,\ldots,p-1\}$ satisfy $\rho_{2}^{i}D\rho_{2}^{-i}\subset \widetilde{H}_{2}$, then we must have $i=0$. This implies an equality
\begin{equation*}
\Delta_{D}+\cI_{D}=\bigoplus_{g\in \langle \rho_{2}\rangle}D^{\vee},
\end{equation*}
and hence the condition (2) is always valid. Therefore, we obtain the desired isomorphism. 

\emph{Case 3.~$n=1$. }
In this case, the assertion follows from Lemma \ref{lem:sshe} (iii), since $\widetilde{H}_{2}=N_{p-1}$ is normal in $P_{p}$. 
\end{proof}

\begin{prop}\label{prop:mxsh}
Let $p>2$ be an odd prime number. Consider a subgroup $\widetilde{H}_{2}:=\langle \eta_{p}\rangle\cdot N_{p-2}$ of $P_{p}$ and an admissible set $\widetilde{\cD}$ of subgroups in $P_{p}$ that contains its maximal subgroup. Then, there is an isomorphism
\begin{equation*}
\sS_{\widetilde{H}_{2},\widetilde{\cD}}^{2}(P_{p},\Ind_{N_{p}}^{P_{p}}\Z)\cong (N_{p}/N_{p-2})^{\vee}. 
\end{equation*}
\end{prop}

\begin{proof}
Take $f\in \sS_{\widetilde{H}_{2},\widetilde{\cD}}^{2}(P_{p},\Ind_{N_{p}}^{P_{p}}\Z)$, and regard it as an element of $N_{p}^{\vee}$. Note that we have
\begin{equation*}
(N_{p}/N_{p-2})^{\vee}\subset \sS_{\widetilde{H}_{2},\widetilde{\cD}}^{2}(P_{p},\Ind_{N_{p}}^{P_{p}}\Z),
\end{equation*}
which follows from Lemma \ref{lem:sshe} (ii) because $\widetilde{H}_{2}\cap P_{p}^{\der}=N_{p-2}$. 

\vspace{3pt}
\emph{Case 1.~$N_{p} \in \widetilde{\cD}$. }
Since Proposition \ref{prop:ppsh} implies $f\!\mid_{N_{p-3}}=0$, it suffices to prove $f(\eta_{p-2})=0$. Note that $\langle \rho_{2}\rangle$ a complete representative of $N_{p}\backslash P_{p}/N_{p}$ in $P_{p}$. Hence, we have
\begin{equation}\label{eq:sscd}
(f\circ \Ad(g^{-1}))_{g\in \langle \rho_{2}\rangle}=(\varphi+\psi_{g})_{g\in \langle \rho_{2}\rangle}
\end{equation}
for some $\varphi \in E^{\vee}$ and $\psi_{g}\in (N_{p}/g\widetilde{H}_{n}g^{-1})^{\vee}$. Consider the $\rho_{2}^{i}$-th factor of the image of $\rho_{2}^{i}\eta_{p}\rho_{2}^{-i}$ under \eqref{eq:sscd} for each $i\in \{0,\ldots,p-1\}$. Then we obtain $\varphi(\rho_{2}^{i}\eta_{p}\rho_{2}^{-i})=f(\eta_{p})$. In particular, one has $\varphi([\eta_{p},\rho_{2}^{i}])=0$ for all $i\in \{0,\ldots,p-1\}$. Furthermore, this implies an equality
\begin{equation}\label{eq:fpsi}
f([\eta_{p},\rho_{2}^{i}])=\psi_{1}([\eta_{p},\rho_{2}^{i}])
\end{equation}
for any $i\in \{0,\ldots,p-1\}$ by considering the factor at $1\in \langle \rho_{2}\rangle$ of the image of $[\eta_{p},\rho_{2}^{i}]$ under \eqref{eq:sscd}. On the other hand, direct computation gives an equality $\eta_{p-2}=[\eta_{p},\rho_{2}^{2}][\eta_{p},\rho_{2}]^{-2}$. Hence, we have
\begin{equation}\label{eq:fpdt}
f(\eta_{p-2})=\psi_{1}(\eta_{p-2}). 
\end{equation}
by \eqref{eq:sscd}. However, the right-hand side of \eqref{eq:fpdt} is zero since $\psi_{1}\in (N_{p}/\widetilde{H}_{2})^{\vee}$. Therefore, we get $f(\eta_{p-2})=0$ as desired. 

\vspace{3pt}
\emph{Case 2.~$N_{p} \notin \widetilde{\cD}$. }
By assumption, there is $i\in \{0,\ldots,p-1\}$ such that $D:=\langle \eta_{p}\rho_{2}^{i}\rangle \cdot N_{p-1}$ lies in $\widetilde{\cD}$. Then, $D\backslash P_{p}/N_{p}$ is singuleton, and hence $\{1\}$ is a complete representative of $D\backslash P_{p}/N_{p}$ in $P_{p}$. Therefore, we obtain an equality
\begin{equation*}
f\!\mid_{D\cap N_{p}}=\varphi+\psi
\end{equation*}
for some $\varphi \in ((D\cap N_{p})/(D^{\der}\cap N_{p}))^{\vee}$ and $\psi \in ((D\cap N_{p})/(D\cap \widetilde{H}_{2}))^{\vee}$. On the other hand, one has $D\cap N_{p}=N_{p-1}$ and $D\cap \widetilde{H}_{2}=N_{p-2}$ by definition. In addition, Lemma \ref{lem:elst} (iii) implies $D^{\der}=N_{p-2}$. In particular, we have
\begin{equation*}
((D\cap N_{p})/(D^{\der}\cap N_{p}))^{\vee}=(N_{p-1}/N_{p-2})^{\vee}=((D\cap N_{p})/(D\cap \widetilde{H}_{2}))^{\vee}. 
\end{equation*}
Hence, we get $f\!\mid_{N_{p-2}}=0$, that is, $f\in (N_{p}/N_{p-2})^{\vee}$. 
\end{proof}

In conclusion, we obtain the following: 

\begin{thm}\label{thm:shpg}
Let $p>2$ be an odd prime number, $G$ a transitive group of degree $p^{2}$, and $H$ a one point stabilizer in $G$. We further assume that $G$ is a $p$-group. 
\begin{enumerate}
\item If $\Sha_{\omega}^{2}(G,J_{G/H})\neq 0$, then we have $G\cong P'_{n}$ for some $n\in \{1,2\}$. 
\item Assume $G\cong P'_{n}$ for some $n\in \{1,2\}$. For an admissible set $\cD$ of subgroups in $G$, there is an isomorphism
\begin{equation*}
\Sha_{\cD}^{2}(G,J_{G/H})\cong 
\begin{cases}
0&\text{if $(C_{p})^{2}<D$ for some $D\in \cD$; }\\
\Z/p&\text{otherwise. }
\end{cases}
\end{equation*}
\end{enumerate}
\end{thm}

\begin{proof}
By Proposition \ref{prop:dwpg}, we may assume that $G$ coincides with $P_{n}$ or $P'_{n}$ for some $n\in \{1,\ldots,p\}$. 

\vspace{3pt}
\emph{Case 1.~$G=P_{n}$ for some $n\in \{1,\ldots,p\}$. }
We may assume $H=H_{n}$. It suffices to prove $\Sha_{\omega}^{2}(G,J_{G/H})=0$. Note that we have $\Sha_{\omega}^{2}(N_{n},J_{N_{n}/H_{n}})=0$ by Proposition \ref{prop:btls}. Hence, Lemma \ref{lem:pnor} gives an exact sequence
\begin{equation*}
0\rightarrow ((N_{n}\cap H_{n}P_{n}^{\der})/H_{n}N_{n}^{\der})^{\vee} \rightarrow \Sel_{H_{n}}^{2}(P_{n},J_{P_{n}/N_{n}})\rightarrow \Sha_{\omega}^{2}(G,J_{G/H})\rightarrow 0. 
\end{equation*}
Note that Lemma \ref{lem:enel} (iii) gives the following: 
\begin{equation*}
((N_{n}\cap H_{n}P_{n}^{\der})/H_{n}N_{n}^{\der})^{\vee}=(N_{n}/H_{n})^{\vee}\cong \Z/p. 
\end{equation*}
On the other hand, we can apply Lemma \ref{lem:sssl} because of Proposition \ref{prop:bgtr}. In particular, we obtain an exact sequence
\begin{equation*}
0\rightarrow (N_{n}/(N_{n}\cap P_{n}^{\der}))^{\vee}\rightarrow \sS_{H_{n}}^{2}(P_{n},\Ind_{N_{n}}^{P_{n}}\Z)\rightarrow \Sel_{H_{n}}^{2}(P_{n},J_{P_{n}/N_{n}})\rightarrow 0. 
\end{equation*}
By Lemma \ref{lem:enel} (iii), we have $(N_{n}/(N_{n}\cap P_{n}^{\der}))^{\vee}=(N_{n}/N_{n-1})^{\vee}\cong \Z/p$. Therefore, the desired assertion is reduced to the following: 
\begin{equation*}
\sS_{H_{n}}^{2}(P_{n},\Ind_{N_{n}}^{P_{n}}\Z)\cong (\Z/p)^{\oplus 2}.
\end{equation*}
However, this follows from Proposition \ref{prop:g1vn}. 

\vspace{6pt}
\emph{Case 2.~$G=P'_{n}$ for some $n\in \{1,\ldots,p-1\}$. }
We may assume $H=H_{n}$. Consider the homomorphism $\pi_{n}\colon P_{p}\twoheadrightarrow P'_{n}$ in Lemma \ref{lem:pmpr} (ii). Pick an admissible set $\widetilde{\cD}$ of subgroups in $P_{p}$ so that $N_{p-n}\subset D$ and $D/N_{p-n}$ is non-cyclic for any $D\in \widetilde{\cD}\setminus \cC_{P_{p}}$. By Proposition \ref{prop:ifts}, it suffices to prove the following: 
\begin{equation}\label{eq:odpp}
\Sha_{\widetilde{\cD}}^{2}(P_{p},J_{P_{p}/\widetilde{H}_{n}})\cong 
\begin{cases}
\Z/p&\text{if $n\leq 2$ and $\widetilde{\cD}=\cC_{P_{n}}$};\\
0&\text{otherwise. }
\end{cases}
\end{equation}
It is clear that $\Sha_{\widetilde{\cD}}^{2}(P_{p},J_{P_{p}/\widetilde{H}_{n}})=0$ if $P_{p}\in \widetilde{\cD}$. Hence, we further assume $P_{p}\notin \widetilde{\cD}$ in the following. By the same argument as Case 1, we obtain exact sequences as follows: 
\begin{gather*}
0\rightarrow ((N_{p}\cap \widetilde{H}_{n}P_{p}^{\der})/\widetilde{H}_{n}N_{p}^{\der})^{\vee} \rightarrow \Sel_{\widetilde{H}_{n},\widetilde{\cD}}^{2}(P_{p},J_{P_{p}/N_{p}})\rightarrow \Sha_{\widetilde{\cD}}^{2}(P_{p},J_{P_{p}/\widetilde{H}_{n}})\rightarrow 0,\\
0\rightarrow (N_{p}/(N_{p}\cap P_{p}^{\der}))^{\vee} \rightarrow \sS_{\widetilde{H}_{n},\widetilde{\cD}}^{2}(P_{p},\Ind_{N_{p}}^{P_{p}}\Z)\rightarrow \Sel_{\widetilde{H}_{n},\widetilde{\cD}}^{2}(P_{p},J_{P_{p}/N_{p}})\rightarrow 0. 
\end{gather*}
Moreover, the following are valid: 
\begin{equation}\label{eq:pech}
(N_{p}/(N_{p}\cap P_{p}^{\der}))^{\vee}\cong \Z/p,\quad
((N_{p}\cap \widetilde{H}_{n}P_{p}^{\der})/\widetilde{H}_{n}N_{p}^{\der})^{\vee}\cong 
\begin{cases}
\Z/p&\text{if }n\geq 2;\\
0&\text{if }n=1. 
\end{cases}
\end{equation}
First, if $n\geq 3$, then we have $\sS_{\widetilde{H}_{n}}^{2}(P_{p},\Ind_{N_{p}}^{P_{p}}\Z)\cong (\Z/p)^{\oplus 2}$ by Proposition \ref{prop:ppsh}. Hence, we obtain \eqref{eq:odpp} in this case by using \eqref{eq:pech}. Second, assume $n=2$. Then, we may assume
\begin{itemize}
\item[(A)] $\widetilde{\cD}=\cC_{P_{p}}$; or 
\item[(B)] $\widetilde{\cD}$ contains a maximal subgroup of $P_{p}$. 
\end{itemize}
Hence, Propositions \ref{prop:ppsh} and \ref{prop:mxsh} imply
\begin{equation*}
\sS_{\widetilde{H}_{2},\widetilde{\cD}}^{2}(P_{p},\Ind_{N_{p}}^{P_{p}}\Z)\cong 
\begin{cases}
(\Z/p)^{\oplus 3}&\text{if (A) holds; }\\
(\Z/p)^{\oplus 2}&\text{if (B) holds.}
\end{cases}
\end{equation*}
Combining this with \eqref{eq:pech}, we obtain \eqref{eq:odpp} in the case $n=2$. Finally, suppose $n=1$, which implies $\cD=\cC_{P_{1}}$ by assumption. Note that $H_{1}=\{1\}$ by definition. Then, we obtain
\begin{equation*}
\sS_{\widetilde{H}_{1}}^{2}(P_{p},J_{P_{p}})\cong (\Z/p)^{\oplus 2}. 
\end{equation*}
by Proposition \ref{prop:ppsh}. Hence, the isomorphism $\Sha_{\omega}^{2}(P_{p},J_{P_{p}/\widetilde{H}_{1}})\cong \Z/p$ follows from \eqref{eq:pech}. 
\end{proof}

\section{Computation of second cohomology groups: General case}\label{sect:cshg}

\begin{prop}\label{prop:nhvn}
Let $p>2$ be an prime number, $G$ be a transitive group of degree $p^{2}$, and $H$ a one point stabilizer in $G$. Denote by $P$ a $p$-Sylow subgroup of $G$. We further assume that $P$ is not isomorphic to $P'_{1}$ or $P'_{2}$. Then, we have
\begin{equation*}
\Sha_{\omega}^{2}(G,J_{G/H})=0. 
\end{equation*}
\end{prop}

\begin{proof}
By Proposition \ref{prop:shpp}, we obtain an injection
\begin{equation*}
\Sha_{\omega}^{2}(G,J_{G/H})\hookrightarrow \Sha_{\omega}^{2}(P,J_{P/(P\cap H)}). 
\end{equation*}
On the other hand, since $P\not\cong P'_{1}$ and $P'\not\cong P'_{2}$, we obtain $\Sha_{\omega}^{2}(P,J_{P/(P\cap H)})=0$. Therefore, the assertion holds. 
\end{proof}

\begin{prop}\label{prop:shaz}
Let $p>2$ be an odd prime number, $G$ a transitive group of degree $p^{2}$, and $H$ a one point stabilizer in $G$. We further assume that there exists a subgroup $G'$ of $G$ with $(G:G')\leq 2$ and transitive groups $G_{1}$ and $G_{2}$ of degree $p$ with $\#G_{1}>p$ or $\#G_{2}>p$ such that $G_{1}\times G_{2}<G'<N_{\fS_{p}}(G_{1})\times N_{\fS_{p}}(G_{2})$. Then, we have
\begin{equation*}
\Sha_{\omega}^{2}(G,J_{G/H})=0. 
\end{equation*}
\end{prop}

\begin{proof}
Since $\ord_{p}(\#G_{i})=\ord_{2}(\#N_{\fS_{p}}(G_{i}))=1$, the homomorphism
\begin{equation*}
\Res_{G/(G_{1}\times G_{2})}\colon \Sha_{\omega}^{2}(G,J_{G/H})\rightarrow \Sha_{\omega}^{2}(G_{1}\times G_{2},J_{G_{1}\times G_{2}/H\cap (G_{1}\times G_{2})})
\end{equation*}
is injective. Hence, we may assume $G=G_{1}\times G_{2}$ with $\#G_{2}>p$. Consider the canonical exact sequence
\begin{equation*}
H^{2}(H,\Z)[p^{\infty}]\xrightarrow{j_{*}} H^{2}(G,J_{G/H})[p^{\infty}]\xrightarrow{\delta} H^{3}(G,\Z)[p^{\infty}]. 
\end{equation*}
Note that $\Sha_{\omega}^{2}(G,J_{G/H})$ is contained in $H^{2}(G,J_{G/H})[p^{\infty}]$, which is a consequence of Proposition \ref{prop:andg}. It suffices to prove $H^{2}(G,J_{G/H})[p^{\infty}]=0$. First, we have $H^{2}(H,\Z)[p^{\infty}]=0$ since $H$ has order coprime to $p$. On the other hand, Proposition \ref{prop:thrt} gives an exact sequence
\begin{equation*}
0 \rightarrow H^{1}(G_{1},G_{2}^{\vee})[p^{\infty}]\rightarrow H_{(G_{1})}^{2}(G,\Q/\Z)[p^{\infty}]\xrightarrow{\Res_{G/G_{2}}} H^{2}(G_{2},\Q/\Z)[p^{\infty}]. 
\end{equation*}
Since $\#G_{2}>p$, Lemma \ref{lem:nadp} implies $\#G_{2}^{\vee}\notin p\Z$. In particular, $H^{1}(G_{1},G_{2}^{\vee})[p^{\infty}]=0$. Furthermore, we have $H^{2}(G_{i},\Q/\Z)[p^{\infty}]=0$ for $i\in \{1,2\}$. Hence, we obtain an equality
\begin{equation*}
H^{2}(G,\Q/\Z)[p^{\infty}]=H_{(G_{1})}^{2}(G,\Q/\Z)[p^{\infty}]=0. 
\end{equation*}
Therefore, one has $H^{3}(G,\Z)[p^{\infty}]=0$, and this implies $H^{2}(G,J_{G/H})[p^{\infty}]$ as desired. 
\end{proof}

\begin{lem}\label{lem:cpij}
Let $p>2$ be an odd prime number, $n\in \{1,2\}$, $G:=P'_{n}$, and $H:=H_{n}$. Take a subgroup $N$ of order $p^{2}$ in $G$ so that $H\cap N=\{1\}$. Then, the composite
\begin{equation*}
\Sha_{\omega}^{2}(G,J_{G/H})\xrightarrow{\delta_{G/H}^{2}} H^{3}(G,\Z)\xrightarrow{\Res_{G/N}}H^{3}(N,\Z),
\end{equation*}
where $\delta_{G/H}^{2}$ is the connecting homomorphism, is injective. 
\end{lem}

\begin{proof}
First, assume $n=1$, then $H=\{1\}$. Hence,
\begin{equation*}
H^{2}(G,\Ind_{H}^{G}\Z)\cong H^{2}(H,\Z)=0
\end{equation*}
by Shapiro's lemma. In particular, $\delta_{G/H}^{2}$ is injective. Now, if we set $N:=G$, then the assertion is clear. Next, assume $n=2$. Then, one has an isomorphism of $N$-lattices $J_{G/H}\cong J_{N}$ since $H\cap N=\{1\}$. Hence, we have a commutative diagram
\begin{equation*}
\xymatrix@C=35pt{
\Sha_{\omega}^{2}(G,J_{G/H})\ar[r]^{\Res_{G/N}}\ar[d]^{\delta_{G/H}^{2}}&\Sha_{\omega}^{2}(N,J_{N})\ar[d]^{\delta_{N}^{2}}\\
H^{3}(G,\Z)\ar[r]^{\Res_{G/N}}&H^{3}(N,\Z). 
}
\end{equation*}
Now, Theorem \ref{thm:shpg} (ii) implies that the top horizontal map $\Res_{G/N}$ is injective. Hence, the assertion follows from the case $n=1$. 
\end{proof}

\begin{lem}[{\cite[Chapter 10, Theorem 7.8 (v)]{Karpilovsky1993}}]\label{lem:bcfx}
Let $G$ be a subgroup of $\GL_{2}(\Fp)$, and consider its action on $(C_{p})^{2}$ induced by the standard representation of $\SL_{2}(\Fp)$. Then, we have
\begin{equation*}
H^{2}((C_{p})^{2},\Q/\Z)^{G}\neq 0
\end{equation*}
if and only if $G<\SL_{2}(\Fp)$. 
\end{lem}

\begin{prop}\label{prop:nszr}
Let $p>2$ be an odd prime number, and $G'$ a subgroup of $\GL_{2}(\Fp)$ that is not contained in $\SL_{2}(\Fp)$. Put
\begin{equation*}
G:=(C_{p})^{2}\rtimes_{\varphi}G',\quad H:=\{1\}\rtimes_{\varphi}G',
\end{equation*}
where $\varphi$ is induced by the standard representation of $\GL_{2}(\Fp)$. Then, we have
\begin{equation*}
\Sha_{\omega}^{2}(G,J_{G/H})=0. 
\end{equation*}
\end{prop}

\begin{proof}
Let $P$ be a $p$-Sylow subgroup of $G$. By Lemma \ref{lem:sytr}, $P$ is a transitive group of degree $p^{2}$, and $P\cap H$ is a one point stabilizer in $P$. Put $N:=(C_{p})^{2}\rtimes \{1\}$, which satisfies $(P\cap H)\cap N=\{1\}$ and $G=N\rtimes G'$. Then, we have a commutative diagram
\begin{equation*}
\xymatrix@C=35pt{
\Sha_{\omega}^{2}(G,J_{G/H})\ar[r]^{\hspace{10pt}\delta_{G/H}^{2}}\ar@{^{(}->}[d]_{\Res_{G/P}}& H^{3}(G,\Z)\ar[r]^{\Res_{G/N}}\ar[d]_{\Res_{G/P}}&H^{3}(N,\Z)^{G'}\ar@{^{(}->}[d]\\
\Sha_{\omega}^{2}(P,J_{P/(P\cap H)})\ar[r]^{\hspace{20pt}\delta_{P/(P\cap H)}^{2}}& H^{3}(P,\Z)\ar[r]^{\Res_{P/N}}&H^{3}(N,\Z)
}
\end{equation*}
Note that the composite of the bottom horizontal maps is injective, which is a consequence of Lemma \ref{lem:cpij}. Moreover, Lemma \ref{lem:bcfx} implies $H^{3}(N,\Z)^{G'}=0$ since $G'$ is not contained in $\SL_{2}(\Fp)$. Therefore, we obtain the desired assertion. 
\end{proof}

\begin{lem}\label{lem:bell}
Let $p$ be a prime number, and $V$ the standard representation of $\SL_{2}(\Fp)$ over $\Fp$. Then, for each $i\in \{1,2\}$, we have
\begin{equation*}
H^{i}(\SL_{2}(\Fp),V)=H^{i}(\SL_{2}(\Fp),V^{\vee})=0. 
\end{equation*}
Here, we regard $V^{\vee}$ as an $\Fp$-representation of $\SL_{2}(\Fp)$ by the contragredient of $V$. 
\end{lem}

\begin{proof}
The assertion $H^{i}(\SL_{2}(\Fp),V)=0$ for $i\in \{1,2\}$ is a consequence of \cite[Table 1]{Bell1978a}. To prove $H^{i}(\SL_{2}(\Fp),V^{\vee})=0$ for $i\in \{1,2\}$, take the standard basis $e_{1},e_{2}$ on $V$, and denote by $e_{1}^{\vee},e_{2}^{\vee}$ its dual basis on $V^{\vee}$. Then, under the isomorphism of $\Fp$-vector spaces
\begin{equation*}
\iota \colon V\xrightarrow{\cong}V^{\vee};\,e_{i}\mapsto (-1)^{i-1}e_{3-i}^{\vee}, 
\end{equation*}
one has a commutative diagram
\begin{equation*}
\xymatrix@C=60pt{
\SL_{2}(\Fp)\times V \ar[r]^{\hspace{25pt}(g,v)\mapsto gv}\ar[d]_{\id \times \iota}^{\cong}& V\ar[d]^{\cong}\\
\SL_{2}(\Fp)\times V^{\vee}\ar[r]^{\hspace{25pt}(g,v^{\vee})\mapsto gv^{\vee}}&V^{\vee}. 
}
\end{equation*}
Hence, \cite[Chap.~VII, \S 5]{Serre1979}, gives an isomorphism
\begin{equation*}
H^{i}(\SL_{2}(\Fp),V^{\vee})\cong H^{i}(\SL_{2}(\Fp),V)
\end{equation*}
for each $i\in \{1,2\}$. Now, the assertion follows from the vanishing of $H^{i}(\SL_{2}(\Fp),V)$. 
\end{proof}

Let $G$ be a finite group. We denote by $\Aut(G)$ its automorphism group, and write $\Inn(G)$ for the inner automorphism of $G$. Define a subgroup of $\Aut(G)$ as follows: 
\begin{equation*}
\Aut(G)^{\circ}:=\{f\in \Aut(G)\mid f\!\mid_{Z(G)}=\id_{Z(G)}\}. 
\end{equation*}
Then, it contains $\Inn(G)$ as a normal subgroup. Hence, we can consider the quotient group
\begin{equation*}
\Out(G)^{\circ}:=\Aut(G)^{\circ}/\Inn(G). 
\end{equation*}
In the following, we study the structure of $\Aut(P'_{2})^{\circ}$. We use the notations in Section \ref{sect:trgp}. By Lemma \ref{lem:pppr} (i), we have $Z(P'_{2})=\langle \eta_{1}\rangle$. Hence, we have $(P'_{2})^{\ad}\cong (C_{p})^{2}$. 

\begin{prop}[{cf.~\cite[Theorem 1 (a)]{Winter1972}}]\label{prop:wint}
Let $p>2$ be an odd prime number, and $P:=P'_{2}$. Then, there is an isomorphism $\overline{\varphi}\colon \SL_{2}(\Fp)\rightarrow \Out(P)^{\circ}$ such that 
\begin{equation}\label{eq:otwt}
\begin{gathered}
\xymatrix{
\SL_{2}(\Fp)\ar[r]^{\overline{\varphi}\hspace{5pt}}\ar@{=}[d]&\Out(P)^{\circ} \ar[d]^{f\mapsto [gZ(P)\mapsto f(g)Z(P)]}\\
\SL_{2}(\Fp)\ar[r]^{\varphi\hspace{5pt}}& \Aut(P^{\ad}),
}
\end{gathered}
\end{equation}
is commutative. Here $\varphi$ is defined as follows: 
\begin{equation*}
\varphi(g)=[\eta_{2}^{x_{1}}\rho_{2}^{y_{2}}Z(P)\mapsto \eta_{2}^{ax_{1}+cx_{2}}\rho_{2}^{bx_{1}+dx_{2}}Z(P)],\quad g=\begin{pmatrix}a&b\\c&d\end{pmatrix}. 
\end{equation*}
\end{prop}

\begin{proof}
Recall Lemma \ref{lem:elst} (iii) that $[\eta_{2},\rho_{2}]=\eta_{1}$. Moreover, we have $Z(P)=\langle \eta_{1}\rangle$ and $\gamma_{2}^{p}=1\in Z(G)$ by Lemma \ref{lem:pppr}. Hence, we can apply the argument in \cite[\S 3]{Winter1972}. By \cite[(3A), (3C)]{Winter1972}, the right vertical homomorphism in \eqref{eq:otwt} is injective, and factors as follows:
\begin{equation*}
\Out(P)^{\circ}\xrightarrow{q} \SL(P^{\ad})\hookrightarrow \Aut(P^{\ad}). 
\end{equation*}
Note that the equality $\Ima(\varphi)=\Ima(q)$ holds. For $g=\begin{pmatrix}a&b\\c&d\end{pmatrix}$, let
\begin{equation*}
\psi(g)\colon P\rightarrow P;\,\eta_{1}^{y}\eta_{2}^{x_{1}}\rho_{2}^{x_{2}}\mapsto \eta_{1}^{y}\eta_{2}^{ax_{1}+cx_{2}}\rho_{2}^{bx_{1}+x_{2}}. 
\end{equation*}
Then, applying \cite[(3F)]{Winter1972} to the transpose of $g$, we obtain that $\psi(g)$ an automorphism of $P$. Moreover, the image of $\psi(g)$ under the right vertical map in \eqref{eq:otwt} equals $\varphi(g)$. Hence, we obtain that $q$ is an isomorphism, and hence the correspondence $g\mapsto \psi(g)$ induces an isomorphism $\overline{\varphi}$. 
\end{proof}

\begin{prop}\label{prop:wntr}
Let $p>2$ be an odd prime number, and $P:=P'_{2}$. Then, there is a commutative diagram
\begin{equation*}
\xymatrix{
(C_{p})^{2}\rtimes_{\varphi} \SL(\Fp)\ar[r]^{\hspace{20pt}\cong}\ar[d]^{\pr_{2}}&\Aut(P)^{\circ}\ar[d]^{f\mapsto [gZ(P)\mapsto f(g)Z(P)]}\\
\SL_{2}(\Fp)\ar[r]& \Aut(P^{\ad}). 
}
\end{equation*}
where $\varphi$ is as in Proposition \ref{prop:wint}. In particular, there is an injective homomorphism
\begin{equation*}
\widetilde{\varphi}\colon \SL_{2}(\Fp)\rightarrow \Aut(P)^{\circ}
\end{equation*}
such that the composite with $\Aut(P)^{\circ} \rightarrow \Aut(P^{\ad})$ coincides with $\varphi$. 
\end{prop}

\begin{proof}
By Proposition \ref{prop:wint}, it suffices to prove $\Aut(P)^{\circ}\cong P^{\ad}\rtimes_{\varphi}\SL_{2}(\Fp)$. Consider the canonical exact sequence
\begin{equation}\label{eq:cdsp}
1\rightarrow \Inn(P)\rightarrow \Aut(P)^{\circ}\rightarrow \SL_{2}(\Fp)\rightarrow 1. 
\end{equation}
Define an isomorphism $\alpha \colon \F_{p}^{\oplus 2}\xrightarrow{\cong}\Inn(P)$ defined as follows: 
\begin{equation*}
\alpha(1,0)=\Ad(\eta_{2}),\quad \alpha(0,1)=\Ad(\rho_{2}). 
\end{equation*}
Then, under the above isomorphism, the action of $\SL_{2}(\Fp)$ on $\Inn(P)$ induced by $\overline{\varphi}$ in Proposition \ref{prop:wint} is equivalent to the standard representation of $\SL_{2}(\Fp)$. On the other hand, if we denote by $V$ the standard representation of $\SL_{2}(\Fp)$, then we have
\begin{equation*}
H^{2}(\SL_{2}(\Fp),V)=0
\end{equation*}
by Lemma \ref{lem:bell}. This implies that \eqref{eq:cdsp} splits, and hence we obtain the desired isomorphism $\Aut(P)^{\circ}\cong P^{\ad}\rtimes_{\varphi}\SL_{2}(\Fp)$. 
\end{proof}

\begin{lem}\label{lem:evlf}
Let $p>2$ be an odd prime number, and put $P:=P'_{2}$. Take a homomorphism
\begin{equation*}
\widetilde{\varphi}\colon \SL_{2}(\Fp)\rightarrow \Aut(P)^{\circ} 
\end{equation*}
in Proposition \ref{prop:wntr}. Assume that $g\in \SL_{2}(\Fp)$ and $x\in P$ satisfy $g(xZ(P))=xZ(P)$. Then, we have $\widetilde{\varphi}(g)(x)=x$. 
\end{lem}

\begin{proof}
Pick $x\in P$. If $x\in Z(P)$, then the assertion is clear. Otherwise, there exist $j_{0},j_{1}\in \Z/p$ such that
\begin{equation*}
\widetilde{\varphi}(g)(x)=x\eta_{1}^{j_{0}},\quad \widetilde{\varphi}(-E_{2})(x)=x^{-1}\eta_{1}^{j_{1}}. 
\end{equation*}
Here $E_{2}$ is the unit matrix of size $2$. Then, we have the following: 
\begin{equation*}
\widetilde{\varphi}(-E_{2})\circ \widetilde{\varphi}(g)(x)=x^{-1}\eta_{1}^{j_{0}+j_{1}},\quad \widetilde{\varphi}(g)\circ \widetilde{\varphi}(-E_{2})(x)=x^{-1}\eta_{1}^{-j_{0}+j_{1}}. 
\end{equation*}
Since $\widetilde{\varphi}(-E_{2})\circ \widetilde{\varphi}(g)=\widetilde{\varphi}(-E_{2})\circ \widetilde{\varphi}(g)$, one has $j_{0}+j_{1}=-j_{0}+j_{1}$. This implies $j_{0}=0$ since $p>2$. Hence, we obtain $\widetilde{\varphi}(g)(x)=x$ as desired. 
\end{proof}

\begin{lem}\label{lem:masl}
Let $p>2$ be an odd prime number. Put $G:=(C_{p})^{2}\rtimes_{\varphi} \SL_{2}(\Fp)$, where $\varphi$ is induced by the standard representation of $\SL_{2}(\Fp)$. Then, there is an isomorphism
\begin{equation*}
H^{2}(G,\Q/\Z)\cong \Z/p. 
\end{equation*}
\end{lem}

\begin{proof}
Put $N:=(C_{p})^{2}\rtimes_{\varphi}\{E_{2}\}$ and $H:=\{1\}\rtimes_{\varphi}\SL_{2}(\Fp)$, which satisfies $G=N\rtimes H$. Denote by $V$ the standard representation of $\SL_{2}(\Fp)$. By \cite[Chapter 16, Theorem 4.5]{Karpilovsky1993}, we have $H^{2}(H,\Q/\Z)=0$. Combining this fact with Proposition \ref{prop:thrt}, we obtain an exact sequence
\begin{equation*}
0\rightarrow H^{1}(H,V^{\vee})\rightarrow H^{2}(G,\Q/\Z)\xrightarrow{\Res_{G/N}}H^{2}(N,\Q/\Z)^{H}\rightarrow H^{2}(H,V^{\vee}). 
\end{equation*}
On the other hand, by Lemma \ref{lem:bell}, one has $H^{1}(H,V^{\vee})=H^{2}(H,V^{\vee})=0$. Consequently, there is an isomorphism
\begin{equation*}
H^{2}(G,\Q/\Z)\cong H^{2}(N,\Q/\Z)^{H}\cong (\Z/p)^{H}. 
\end{equation*}
Hence, the assertion follows from Lemma \ref{lem:bcfx}. 
\end{proof}

\begin{prop}\label{prop:shpr}
Let $p$ be a prime number. Put
\begin{equation*}
G:=(C_{p})^{2}\rtimes_{\varphi} \SL_{2}(\Fp),\quad H:=\{1\}\rtimes_{\varphi}\SL_{2}(\Fp),
\end{equation*}
where $\varphi$ is induced by the standard representation of $\SL_{2}(\Fp)$. Then, there is an isomorphism
\begin{equation*}
\Sha_{\omega}^{2}(G,J_{G/H})\cong \Z/p. 
\end{equation*}
\end{prop}

\begin{proof}
Take a homomorphism $\widetilde{\varphi} \colon \SL_{2}(\Fp)\rightarrow \Aut(P)^{\circ}$ in Proposition \ref{prop:wntr}. Let
\begin{equation*}
\widetilde{G}:=P\rtimes_{\widetilde{\varphi}}\SL_{2}(\Fp),\quad \widetilde{Z}:=Z(P)\rtimes_{\widetilde{\varphi}}\{E_{2}\}. 
\end{equation*}
Then, the natural inclusion $\widetilde{Z}\hookrightarrow \widetilde{G}$ induces a central extension
\begin{equation}\label{eq:easl}
1\rightarrow \widetilde{Z} \rightarrow \widetilde{G}\rightarrow G \rightarrow 1. 
\end{equation}
Moreover, one has a commutative diagram
\begin{equation*}
\xymatrix{
H^{1}(G,\Q/\Z)\ar[r]\ar[d]^{\cong}&H^{1}(\widetilde{G},\Q/\Z)\ar[r]\ar[d]^{\cong}&H^{1}(\widetilde{Z},\Q/\Z)\ar[d]^{\cong}\\
\SL_{2}(p)^{\vee} \ar@{=}[r]&\SL_{2}(p)^{\vee}\ar[r]&\Z/p. 
}
\end{equation*}
Hence, the transgression map $H^{1}(\widetilde{Z},\Q/\Z)\rightarrow H^{2}(G,\Q/\Z)$ is injective. On the other hand, we have $H^{2}(G,\Q/\Z)\cong \Z/p$ by Lemma \ref{lem:masl}, and hence $H^{1}(\widetilde{Z},\Q/\Z)\rightarrow H^{2}(G,\Q/\Z)$ is surjective. Therefore, \eqref{eq:easl} is a generalized representation group of $G$. 

We use Corollary \ref{cor:drmn} to give the desired isomorphism. It suffices to prove
\begin{itemize}
\item[(a)] $\widetilde{H}\cap \widetilde{G}^{\der}=Z(P)\rtimes_{\widetilde{\varphi}} \SL_{2}(\Fp)^{\der}$; and
\item[(b)] $\Phi^{\widetilde{G}}(\widetilde{H})=\{1\}\rtimes_{\widetilde{\varphi}}\SL_{2}(\Fp)^{\der}$. 
\end{itemize}
The claim (a) is clear. To demonstrate (b), consider elements $z_{1},z_{2}\in Z(P)$, $x\in P$ and $g,h_{1},h_{2}\in \SL_{2}(\Fp)$ that satisfy
\begin{equation*}
(x\rtimes g)(z_{1}\rtimes h_{1})=(z_{2}\rtimes h_{2})(x\rtimes g). 
\end{equation*}
This can be rephrased as two equalities as follows: 
\begin{equation*}
xz_{1}=\widetilde{\varphi}(h_{2})(x)z_{2},\quad gh_{1}=h_{2}g. 
\end{equation*}
In particular, one has $\widetilde{\varphi}(h_{2})(xZ(P))=xZ(P)$. Hence, we have $\widetilde{\varphi}(h_{2})(x)=x$ by Lemma \ref{lem:evlf}. This implies $z_{1}=z_{2}$, and therefore we obtain
\begin{equation*}
(z_{1}\rtimes h_{1})^{-1}(z_{2}\rtimes h_{2})=1\rtimes [h_{1},g].  
\end{equation*}
Consequently, we have $\Phi^{\widetilde{G}}(\widetilde{H})<\{1\}\rtimes \SL_{2}(\Fp)^{\der}$. On the other hand, it is clear that $\Phi^{\widetilde{G}}(\widetilde{H})$ contains $\{1\}\rtimes_{\widetilde{\varphi}}\SL_{2}(\Fp)^{\der}$, and hence the proof is complete. 
\end{proof}

\begin{prop}\label{prop:slnz}
Let $p$ be a prime number, and $G^{\dagger}$ a subgroup of $\SL_{2}(\Fp)$. Put
\begin{equation*}
G:=(C_{p})^{2}\rtimes_{\varphi}G^{\dagger},\quad H:=\{1\}\rtimes_{\varphi}G^{\dagger},
\end{equation*}
where $\varphi$ is induced by the standard representation of $\SL_{2}(\Fp)$. Then, there is an isomorphism
\begin{equation*}
\Sha_{\omega}^{2}(G,J_{G/H})\cong \Z/p. 
\end{equation*}
\end{prop}

\begin{proof}
Put 
\begin{equation*}
G':=(C_{p})^{2}\rtimes_{\varphi} \SL_{2}(\Fp),\quad H':=\{1\}\rtimes_{\varphi}\SL_{2}(\Fp), 
\end{equation*}
where $\varphi$ is induced by the standard representation of $\SL_{2}(\Fp)$. Then, we obtain the natural inclusion $G<G'$, and one has $H=G'\cap H$. Fix $p$-Sylow subgroups $P$ and $P'$ of $G$ and $G'$, respectively, so that $P<P'$. Then, there is an isomorphism $P\cong P'_{2}$. Then, one has a commutative diagram
\begin{equation*}
\xymatrix@C=35pt{
\Sha_{\omega}^{2}(G',J_{G'/H'})\ar[r]^{\Res_{G'/G}}\ar[d]_{\Res_{G'/P'}}& 
\Sha_{\omega}^{2}(G,J_{G/H})\ar[d]^{\Res_{G/P}}\\
\Sha_{\omega}^{2}(P',J_{P'/(P'\cap H')})\ar[r]^{\Res_{P'/P}}&\Sha_{\omega}^{2}(P,J_{P/(P\cap H)}). 
}
\end{equation*}
The bottom horizontal map is injective, which follows from Lemma \ref{lem:cpij}. Moreover, the vertical maps are also injective. Therefore, we obtain that the top horizontal map is injective. On the other hand, we have $\Sha_{\omega}^{2}(G,J_{G/H})\cong \Z/p$ by Proposition \ref{prop:shpr}. In addition, by Therem \ref{thm:shpg} (i), one has an isomorphism $\Sha_{\omega}^{2}(P,J_{P/(P\cap H)})\cong \Z/p$. Therefore, we obtain the desired isomorphism. 
\end{proof}

As a summarization of this section, we obtain the structure of $\Sha_{\cD}^{2}(G,J_{G/H})$ in the case $(G:H)=p^{2}$ with $p$ an odd prime. 

\begin{thm}\label{thm:mtsh}
Let $p>2$ be an odd prime number, $G$ a transitive group of degree $p^{2}$, and $H$ a one point stabilizer in $G$. 
\begin{enumerate}
\item If $\Sha_{\omega}^{2}(G,J_{G/H})\neq 0$, then the following holds: 
\begin{itemize}
\item[($\ast$)] there exists a subgroup $G^{\dagger}$ of $\SL_{2}(\Fp)$ such that 
\begin{equation*}
G\cong (C_{p})^{2}\rtimes_{\varphi} G^{\dagger},\quad 
H\cong \{1\}\rtimes_{\varphi} G^{\dagger}, 
\end{equation*}
where $\varphi$ is induced by the standard representation of $\SL_{2}(\Fp)$ over $\Fp$. 
\end{itemize}
\item Assume that \emph{($\ast$)} is valid. For an admissible set $\cD$ of subgroups in $G$, there is an isomorphism
\begin{equation*}
\Sha_{\cD}^{2}(G,J_{G/H})\cong 
\begin{cases}
0&\text{if $(C_{p})^{2}<D$ for some $D\in \cD$; }\\
\Z/p &\text{otherwise. }
\end{cases}
\end{equation*}
\end{enumerate}
\end{thm}

\begin{proof}
We denote by $P$ a $p$-Sylow subgroup of $G$. 

(i): If $P$ is not isomorphic to $P'_{1}\cong (C_{p})^{2}$ or $P'_{2}$, then we have $\Sha_{\omega}^{2}(G,J_{G/H})=0$ by Proposition \ref{prop:nhvn}. Otherwise, by Proposition \ref{prop:crdw} (i), one of the following is valid: 
\begin{enumerate}
\item[(a)] there exists a subgroup $G'$ of $G$ with $(G:G')\leq 2$ and transitive groups $G_{1}$ and $G_{2}$ of degree $p$ with $\#G_{1}>p$ or $\#G_{2}>p$ such that $G_{1}\times G_{2}<G'<N_{\fS_{p}}(G_{1})\times N_{\fS_{p}}(G_{2})$; 
\item[(b)] there is a subgroup $G'$ of $\GL_{2}(\Fp)$ such that $G\cong (C_{p})^{2}\rtimes G'$ and $H\cong \{1\}\rtimes G'$. 
\end{enumerate}
If (a) holds, then we have $\Sha_{\omega}^{2}(G,J_{G/H})=0$ by Proposition \ref{prop:shaz}. If (b) holds, the assertion follows from Propositions \ref{prop:nszr} and \ref{prop:slnz}. 

(ii): By ($\ast$), we have $P\cong P'_{1}$ or $P\cong P'_{2}$. First, assume that $\cD$ contains a subgroup $D$ of $G$ with $(C_{p})^{2}<D$. Then $\cD_{P}$ is so. Moreover, Proposition \ref{prop:shpp} gives an injection
\begin{equation*}
\Sha_{\cD}^{2}(G,J_{G/H})\hookrightarrow \Sha_{\cD_{\cap P}}^{2}(P,J_{P/(P\cap H)}). 
\end{equation*}
Now, we have $\Sha_{\cD_{\cap P}}^{2}(P,J_{P/(P\cap H)})=0$ by Theorem \ref{thm:shpg} (ii), and hence $\Sha_{\cD}^{2}(G,J_{G/H})=0$. Otherwise, $\cD_{G,p}$ is the set of all cyclic $p$-groups. Consequently, one has
\begin{equation*}
\Sha_{\cD}^{2}(G,J_{G/H})=\Sha_{\cD_{G,p}}^{2}(G,J_{G/H})=\Sha_{\omega}^{2}(G,J_{G/H})
\end{equation*}
by Proposition \ref{prop:rdpp}. Now, (i) implies $\Sha_{\omega}^{2}(G,J_{G/H})\cong \Z/p$, and hence
\begin{equation*}
\Sha_{\cD}^{2}(G,J_{G/H})\cong \Z/p. 
\end{equation*}
This completes the proof. 
\end{proof}

\section{Application to norm one tori}\label{sect:pfch}

\subsection{Norm one tori and their cohomological invariants}

Let $k$ be a field, and fix a separable closure $k^{\sep}$ of $k$. We denote by $\G_{m}$ the multiplicative group scheme. Recall that a \emph{torus} (or, an \emph{algebraic torus}) over $k$ is an algebraic $k$-group $T$ that admits an isomorphism
\begin{equation*}
T\otimes_{k}k^{\sep}\cong \G_{m,k^{\sep}}^{d}
\end{equation*} 
for some non-negative integer $d$. For a torus $T$ over $k$, the \emph{character group} of $T$ is defined as follows: 
\begin{equation*}
X^{*}(T):=\Hom_{k^{\sep}\text{-groups}}(T\otimes_{k}k^{\sep},\G_{m,k^{\sep}}). 
\end{equation*}
It is a finite free abelian group equipped with a continuous action of the absolute Galois group of $k$ (with respect to the discrete topology on $X^{*}(T)$). Hence, $X^{*}(T)$ is a $\Gal(\widetilde{K}/k)$-lattice for some finite Galois extension $\widetilde{K}/k$. 

\vspace{6pt}
Let $X$ be a proper smooth variety over $k$. We denote by $\Br(X)$ the Brauer group of $X$, that is, 
\begin{equation*}
\Br(X):=H_{\et}^{2}(X,\G_{m}). 
\end{equation*}

\begin{prop}[{\cite[Theorem 9.5 (ii)]{ColliotThelene1987}, \cite[Theorem 2.3, Theorem 2.4]{BayerFluckiger2020}}]\label{prop:cfcs}
Let $T$ be a torus over a field $k$ that splits over a finite Galois extension $\widetilde{K}$ of $k$. Denote by $G$ the Galois group of $\widetilde{K}/k$. Take a smooth compactification $X$ of $T$ over $k$, and put $\overline{X}:=X\otimes_{k}k^{\sep}$. Then there is an isomorphism
\begin{equation*}
\Br(X)/\Br(k)\cong H^{1}(k,\Pic(\overline{X}))\cong \Sha_{\omega}^{2}(G,X^{*}(T)). 
\end{equation*}
\end{prop}

Note that smooth compactifications exist for all tori over $k$. It follows from Hironaka (\cite{Hironaka1964}) if $k$ has characteristic zero. On the other hand, if $k$ has positive characteristic, the assertion is a consequence of Colliot-Th{\'e}l{\`e}ne--Harari--Skorobogatov (\cite{ColliotThelene2005}). 

\vspace{6pt}
For a finite separable field extension $K/k$, put
\begin{equation*}
T_{K/k}:=\{t\in \Res_{K/k}\G_{m,K}\mid \N_{K/k}(t)=1\},
\end{equation*}
which is an $k$-torus of dimension $[K:k]-1$. It is called the \emph{norm one torus} attached to $K/k$. 

A group-theoretic description of $X^{*}(T_{K/k})$ is given as follows. Let $\widetilde{K}$ be a finite Galois extension of $k$ that contains $K$. Set $G:=\Gal(\widetilde{K}/k)$ and $H:=\Gal(\widetilde{K}/K)$. Then, there is an isomorphism of $G$-lattices
\begin{equation*}
X^{*}(T_{K/k})\cong J_{G/H}. 
\end{equation*}

\begin{thm}\label{thm:hopc}
Let $p>2$ be an odd prime number. Consider a finite separable field extension $K/k$ of degree $p^{2}$ with its Galois closure $\widetilde{K}/k$. Put $G:=\Gal(\widetilde{K}/k)$ and $H:=\Gal(\widetilde{K}/k)$. Take a smooth compactification $X$ of $T_{K/k}$, and put $\overline{X}:=X\otimes_{k}k^{\sep}$. Then the following are equivalent: 
\begin{enumerate}
\item $H^{1}(k,\Pic(\overline{X}))\neq 0$; 
\item $H^{1}(k,\Pic(\overline{X}))\cong \Z/p$; 
\item there exists a subgroup $G^{\dagger}$ of $\SL_{2}(\Fp)$ such that
\begin{equation*}
G\cong (C_{p})^{2}\rtimes_{\varphi} G^{\dagger},\quad H\cong \{1\}\rtimes_{\varphi} G^{\dagger}, 
\end{equation*}
where $\varphi$ is induced by the standard representation of $\SL_{2}(\Fp)$ over $\Fp$. 
\end{enumerate}
\end{thm}

\begin{proof}
By Proposition \ref{prop:cfcs}, there is an isomorphism $H^{1}(k,\Pic(\overline{X}))\cong \Sha_{\omega}^{2}(G,J_{G/H})$. Hence, the assertion follows from Theorem \ref{thm:mtsh}. 
\end{proof}

\subsection{Hasse norm principle and weak approximation}\label{ssec:pfhn}

Here, let $k$ be a global field, and fix a separable closure $k^{\sep}$ of $k$. For a finite Galois extension $\widetilde{K}/k$, write $\Sigma_{\widetilde{K}}$ for the set of places of $\widetilde{K}$. For each $v \in \Sigma_{\widetilde{K}}$, denote by $D_{v}$ the decomposition group of $\widetilde{K}/k$ at $v$. 

For a finite separable field extension $K/k$, put
\begin{equation*}
\Sha(K/k):=(k^{\times}\cap \N_{K/k}(\A_{K}^{\times}))/\N_{K/k}(K^{\times}),
\end{equation*}
where $\A_{K}^{\times}$ is the id{\`e}le group of $K$. We say that the \emph{Hasse norm principle holds for $K/k$} if
\begin{equation*}
\Sha(K/k)=1. 
\end{equation*}

Let $T$ be a torus over $k$. Then, define the \emph{Tate--Shafarevich group} of $T$ as follows:  
\begin{equation*}
\Sha^{1}(k,T):=\Ker\left(H^{1}(k,T)\xrightarrow{(\Res_{k_{v}/k})_{v}}\prod_{v\in \Sigma_{k}}H^{1}(k_{v},T)\right),
\end{equation*}
Here, $\Res_{k_v/k}\colon H^{1}(k,T)\rightarrow H^{1}(k_{v},T)$ denotes the restriction map for each $v\in \Sigma_{k}$. 

\begin{prop}[{\cite[Theorem 5]{Voskresenskii1969}}]\label{prop:vskr}
Let $k$ be a global field, and $T$ a torus over $k$. Then there is an exact sequence
\begin{equation}\label{eq:vsex}
0\rightarrow A_{k}(T)\rightarrow H^{1}(k,\Pic(\overline{X}))^{\vee}\rightarrow \Sha^{1}(k,T)\rightarrow 0, 
\end{equation}
where $X$ is a smooth compactification of $T$ over $k$, $\overline{X}:=X\otimes_{k}k^{\sep}$ and $A_{k}(T)$ is the quotient of $\prod_{v\in \Sigma_{k}}T(k_{v})$ by the closure of $T(k)$. 
\end{prop}

The following is essentially the same as the Poitou--Tate duality (\cite[(8.6.8) Proposition]{Neukirch2000}). 

\begin{prop}[{cf.~\cite[Proposition 2.22]{Oki2025a}}]\label{prop:almt}
Let $k$ be a global field, and $T$ a $k$-torus which splits over a finite Galois extension $\widetilde{K}$ of $k$. Put $G:=\Gal(\widetilde{K}/k)$. Denote by $\cD$ the set of decomposition groups of $\widetilde{K}/k$, which is an admissible set of subgroups of $G$. Then there is an isomorphism
\begin{equation*}
\Sha^{1}(k,T)\cong \Sha_{\cD}^{2}(G,X^{*}(T))^{\vee}. 
\end{equation*}
\end{prop}

\begin{cor}\label{cor:onot}
Let $K/k$ be a finite separable field extension with Galois closure $\widetilde{K}/k$. Put $G:=\Gal(\widetilde{K}/k)$ and $H:=\Gal(\widetilde{K}/K)$. We denote by $\cD$ the set of decomposition groups of $\widetilde{K}/k$. Then there is an isomorphism
\begin{equation*}
\Sha(K/k)\cong \Sha_{\cD}^{2}(G,J_{G/H})^{\vee}. 
\end{equation*}
\end{cor}

\begin{proof}
By Ono's theorem (\cite[p.~70]{Ono1963}), there is an isomorphism $\Sha(K/k)\cong \Sha^{1}(k,T_{K/k})$. Hence, the assertion follows from Proposition \ref{prop:almt}. 
\end{proof}

\begin{thm}\label{thm:hnps}
Let $p>2$ be an odd prime number, and $k$ a global field. Consider a finite separable field extension $K/k$ of degree $p^{2}$ with Galois closure $\widetilde{K}/k$. Put $G:=\Gal(\widetilde{K}/k)$ and $H:=\Gal(\widetilde{K}/K)$. 
\begin{enumerate}
\item If $\Sha(K/k)\neq 1$, then the following is valid: 
\begin{itemize}
\item[($\ast$)] there exists a subgroup $G^{\dagger}$ of $\SL_{2}(\Fp)$ such that 
\begin{equation*}
G\cong (C_{p})^{2}\rtimes_{\varphi} G^{\dagger},\quad 
H\cong \{1\}\rtimes_{\varphi} G^{\dagger}, 
\end{equation*}
where $\varphi$ is induced by the standard representation of $\SL_{2}(\Fp)$. 
\end{itemize}
\item If \emph{($\ast$)} holds, then there is an isomorphism
\begin{equation*}
\Sha(K/k)\cong 
\begin{cases}
1&\text{if a decomposition group of $\widetilde{K}/k$ contains $(C_{p})^{2}$; }\\
\Z/p &\text{otherwise. }
\end{cases}
\end{equation*}
\end{enumerate}
\end{thm}

\begin{proof}
(i): If $\Sha(K/k)\neq 1$, then Corollary \ref{cor:onot} implies $\Sha_{\omega}^{2}(G,J_{G/H})\neq 0$. Hence, the assertion follows from Theorem \ref{thm:mtsh} (i). 

(ii): Let $\widetilde{K}/k$ be the Galois closure of $K/k$, and put $G:=\Gal(\widetilde{K}/k)$ and $H:=\Gal(\widetilde{K}/K)$. Denote by $\cD$ the set of decomposition groups of $\widetilde{K}/k$. Then, Corollary \ref{cor:onot} gives an isomorphism
\begin{equation*}
\Sha(K/k)\cong \Sha_{\cD}^{2}(G,J_{G/H})^{\vee}. 
\end{equation*}
Therefore, the assertion follows from Theorem \ref{thm:mtsh} (ii). 
\end{proof}

\begin{thm}\label{thm:waps}
Let $p>2$ be an odd prime number, and $k$ a global field. Consider a finite separable field extension $K/k$ of degree $p^{2}$ with Galois closure $\widetilde{K}/k$. Put $G:=\Gal(\widetilde{K}/k)$ and $H:=\Gal(\widetilde{K}/K)$. 
\begin{enumerate}
\item If $A_{k}(T_{K/k})\neq 1$, then \emph{($\ast$)} in Theorem \ref{mth1} is valid. 
\item If \emph{($\ast$)} holds, then there is an isomorphism
\begin{equation*}
A_{k}(T_{K/k})\cong 
\begin{cases}
\Z/p &\text{if a decomposition group of $\widetilde{K}/k$ contains $(C_{p})^{2}$; }\\
1 &\text{otherwise. }
\end{cases}
\end{equation*}
\end{enumerate}
\end{thm}

\begin{proof}
(i): The proof is the same as Theorem \ref{thm:hnps} (i). 

(ii): Let $\widetilde{K}/k$ be the Galois closure of $K/k$, and put $G:=\Gal(\widetilde{K}/k)$ and $H:=\Gal(\widetilde{K}/K)$. Denote by $\cD$ the set of decomposition groups of $\widetilde{K}/k$. Then, Corollary \ref{cor:onot} gives an isomorphism
\begin{equation*}
\Sha(K/k)\cong \Sha_{\cD}^{2}(G,J_{G/H})^{\vee}. 
\end{equation*}
Combining this with Proposition \ref{prop:vskr}, we obtain an isomorphism
\begin{equation*}
A_{k}(T_{K/k})\cong \Coker(\Sha_{\cD}^{2}(G,J_{G/H})\rightarrow \Sha_{\omega}^{2}(G,J_{G/H}))^{\vee}. 
\end{equation*}
Therefore, the assertion follows from Theorem \ref{thm:mtsh} (ii). 
\end{proof}

Finally, we give an alternative proof of a theorem Drakokhrust and Platonov (\cite{Drakokhrust1987}). Recall \cite[p.~451, l.~11--13, Definition]{Schacher1968} that a finite extension $K/k$ of a number field is \emph{$k$-adequate} if there is a central division algebra over $k$ that contains $K$ as a maximal commutative subfield. 

\begin{thm}[{cf.~\cite[Theorem 4]{Drakokhrust1987}}]\label{thm:dpes}
Let $p>2$ be an odd prime number, and $K/k$ a finite extension of a global field of degree $p^{2}$. If $K$ is $k$-adequate, then the Hasse norm principle holds for $K/k$. 
\end{thm}

\begin{proof}
Let $\widetilde{K}/k$ be the Galois closure of $K/k$, and put $G:=\Gal(\widetilde{K}/k)$ and $H:=\Gal(\widetilde{K}/K)$. Pick a $p$-Sylow subgroup $P$ of $G$. If $P\not\cong P'_{1}$ and $P\not\cong P'_{2}$, then the assertion follows from Theorem \ref{thm:hnps} (i). In the following, suppose that $P\cong P'_{1}$ or $P\cong P'_{2}$ holds. By \cite[Lemma 4]{Bartels1981b}, there exists a place of $\widetilde{K}$ at which the decomposition group $D$ in $G$ satisfies 
\begin{equation}\label{eq:dbcs}
P=(P\cap D)\cdot (P\cap H). 
\end{equation}
Here the right-hand side of \eqref{eq:dbcs} is the equivalence class of $1\in P$ in $(P\cap D)\backslash P/(P\cap H)$. Since $(P:P\cap H)=p^{2}$, we obtain $\#(P\cap D)\in p^{2}\Z$. This implies $(C_{p})^{2}<D$ by Lemma \ref{lem:pppr} (ii). Hence, the assertion follows from Theorem \ref{thm:hnps} (ii). 
\end{proof}


\begin{thebibliography}{99}
\bibitem[Bar81a]{Bartels1981a}
H.~J.~Bartels, \emph{Zur Arithmetik von Konjugationsklassen in algebraischen Gruppen}, J.~Alg.~\textbf{70} (1981), 179--199. 
\bibitem[Bar81b]{Bartels1981b}
H.-J.~Bartels, \emph{Zur Arithmetik von Diedergruppenerweiterungun}, Math.~Ann.~\textbf{256} (1981), 465--473. 
\bibitem[BP20]{BayerFluckiger2020}
E.~Bayer-Fluckiger, R.~Parimala, \emph{On unramified Brauer groups of torsors over tori}, Documenta Math.~\textbf{25} (2020), 1263--1284. 
\bibitem[Bel78]{Bell1978a}
G.~W.~Bell, \emph{On the cohomology of the finite special linear groups, I}, J.~Alg.~\textbf{54} (1978), 216--238. 
\bibitem[BT82]{Beyl1982}
F.~R.~Beyl, J.~Tappe, \emph{Group extensions, representations, and the Schur multiplicator}, Lecture Notes in Mathematics,
958, Springer-Verlag, Berlin-New York, 1982.
\bibitem[CS77]{ColliotThelene1977}
J.-L.~Colliot--Th{\'e}l{\`e}ne, J.-J.~Sansuc, \emph{La $R$-{\'e}quivalence sur les tores}, Ann.~sci.~de l'{\'E}.N.S.~$4^{e}$ s{\'e}rie, \textbf{10} no.~2 (1977), 175--229. 
\bibitem[CS87]{ColliotThelene1987}
J.-L.~Colliot-Th{\'e}l{\`e}ne, J.-J.~Sansuc, \emph{Principal homogeneous spaces under flasque tori: applications}, J.~Algebra \textbf{106} (1987), no.~1, 148--205. 
\bibitem[CHS05]{ColliotThelene2005}
J.-L.~Colliot-Th{\'e}l{\`e}ne, D.~Harari, A.~N.~Skorobogatov, \emph{Compactification {\'e}quivariante d’un tore (d’apr{\`e}s Brylinski et K{\"u}nnemann)}, Expo.~Math., \textbf{23} (2005), 161--170.
\bibitem[DM96]{Dixon1996}
J.~D.~Dixon, B~Mortimer, \emph{Permutation groups}, Grad.~Texts in Math.~\textbf{163}, Springer-Verlag, New York, 1996. 
\bibitem[DW02]{Dobson2002}
E.~Dobson, D.~Witte, \emph{Transitive permutation groups of prime-squared degree}, J.~Alg.~Comb.~\textbf{16} (2002), 43--69. 
\bibitem[DP87]{Drakokhrust1987}
Y.~A.~Drakokhrust, V.~P.~Platonov, \emph{The Hasse norm principle for algebraic number fields} (Russian) Izv.~Akad.~Nauk SSSR Ser.~Mat.~\textbf{50} (1986), 946--968; translation in Math.~USSR-Izv.~\textbf{29} (1987), 299--322. 
\bibitem[Dra89]{Drakokhrust1989}
Y.~A.~Drakokhrust, \emph{On the complete obstruction to the Hasse principle}, Amer.~Math.~Soc.~Transl.~(2) \textbf{143} (1989), 299--322. 
\bibitem[EM75]{Endo1975}
S.~Endo, T.~Miyata, \emph{On a classification of the function fields of algebraic tori}, Nagoya Math.~J.~\textbf{56} (1975), 85--104. 
\bibitem[FJ20]{FernandezAlcober2020}
G.~A.~Fern{\'a}ndez-Alcober, U.~Jezernik, \emph{Bogomolov multipliers of $p$-groups of maximal class}, Quart.~J. Math.~\textbf{71} (2020), 123--138. 
\bibitem[Gur78]{Gurak1978}
S.~Gurak, \emph{On the Hasse norm principle}, J.~Angew.~Math.~\textbf{299}/\textbf{300} (1978), 16--27. 
\bibitem[HHY20]{Hasegawa2020}
S.~Hasegawa, A.~Hoshi, A.~Yamasaki, \emph{Rationality problem for norm one tori in small dimensions}, Math.~Comp.~\textbf{89} (2020), no.~322, 923--940. 
\bibitem[HKO25]{Hasegawa2025}
S.~Hasegawa, K.~Kanai, Y.~Oki, \emph{The rationality problem for multinorm one tori}, preprint, arXiv:2504.04078, 2025. 
\bibitem[Has31]{Hasse1931}
H.~Hasse, \emph{Beweis eines Satzes und Wiederlegung einer Vermutung {\"u}ber das allgemeine Normenrestsymbol}, Nschrichten von der Gesellschaft der Wissenschaften zu G{\"o}ttingen, Mathematisch-Physikalische Klasse (1931), 64--69.  
\bibitem[Hir64]{Hironaka1964}
H.~Hironaka, \emph{Resolution of singularities of an algebraic variety over a field of characteristic zero. I, II}, Ann.~Math., (2) \textbf{79} (1964), 109--203; 205--326. 
\bibitem[HKY22]{Hoshi2022}
A.~Hoshi, K.~Kanai, A.~Yamasaki, \emph{Norm one tori and Hasse norm principle}, Math.~Comput.~\textbf{91} (2022), 2431--2458. 
\bibitem[HKY23]{Hoshi2023}
A.~Hoshi, K.~Kanai, A.~Yamasaki, \emph{Norm one tori and Hasse norm principle II: degree $12$ case}, J.~Number Theory \textbf{245} (2023), 84--110. 
\bibitem[HKY25]{Hoshi2025}
A.~Hoshi, K.~Kanai, A.~Yamasaki, \emph{Norm one tori and Hasse norm principle III: Degree $16$ case}, J.~Alg.~\textbf{666} (2025), no.~15, 794--820. 
\bibitem[HY25a]{Hoshi2025a}
A.~Hoshi, A.~Yamasaki, \emph{Hasse norm principle for metacyclic extensions with trivial Schur multiplier}, preprint, arXiv:2503.14365, 2025. 
\bibitem[HY25b]{Hoshi2025b}
A.~Hoshi, A.~Yamasaki, \emph{Hasse norm principle for Heisenberg extensions of degree $p^{3}$}, preprint, arXiv:2503.15408, 2025. 
\bibitem[Kar87]{Karpilovsky1987}
G.~Kalpilovsky, \emph{The Schur multiplier}, London Mathematical Society Monographs, New Series 2, The Claredon Press, Oxford University Press, New York, 1987. 
\bibitem[Kar93]{Karpilovsky1993}
G.~Karpsilovskii, \emph{Group representations.~Vol.~2}, North-Holland Math.~Stud., 177, North-Holland Publishing Co., Amsterdam, 1993. 
\bibitem[Kun84]{Kunyavskii1984}
B.~E.~Kunyavskii, \emph{Arithmetic properties of three-dimensional algebraic tori}, in: Integral Lattices and Finite Linear Groups, Zap.~Nauchn.~Sem.~Leningrad.~Otdel.~Mat.~Inst.~Steklov.~(LOMI) \textbf{116} (1982), 102--107, 163 (Russian); translation in J.~Sov.~Math.~\textbf{26} (1984), 1898--1901. 
\bibitem[Mac20]{Macedo2020}
A.~Macedo, \emph{The Hasse norm principle for $A_{n}$-extensions}, J.~Number Theory \textbf{211} (2020), 500--512. 
\bibitem[MN22]{Macedo2022}
A.~Macedo, R.~Newton, \emph{Explicit methods for the Hasse norm principle and applications to $A_{n}$ and $S_{n}$-extensions}, Math.~Proc.~Phil.~Soc.~\textbf{172}, Issue 3 (2022), 489--529. 
\bibitem[NSW00]{Neukirch2000}
J.~Neukirch, A.~Schmidt, K.~Wingberg, \emph{Cohomology of number fields}, Springer Verlag, 2000. 
\bibitem[Oki25]{Oki2025a}
Y.~Oki, \emph{The Hasse norm principle for some extensions of degree having square-free prime factors}, preprint, arXiv:2504.19453, 2025. 
\bibitem[Ono63]{Ono1963}
T.~Ono, \emph{On Tamagawa numbers of algebraic tori}, Ann.~Math. (2) \textbf{78} (1963), 47--73. 
\bibitem[PR94]{Platonov1994}
V.~P.~Platonov, Rapinchuck, \emph{Algebraic groups and number theory}, Translated from the 1991 Russian original by Rachel
Rowen, Pure and applied mathematics, 139, Academic Press, 1994.
\bibitem[San81]{Sansuc1981}
J.-J.~Sansuc, \emph{Groupe de Brauer et arithm{\'e}tique des groupes alg{\'e}briques lin{\'e}aires sur
un corps de nombres}, J.~Reine Angew.~Math.~\textbf{327} (1981), 12--80. 
\bibitem[Sch68]{Schacher1968}
M.~M.~Schacher, \emph{Subfields of division rings, I}, J.~Alg.~\textbf{9} (1968), 451--477. 
\bibitem[Sch40]{Scholz1940}
A. Scholz, \emph{Totale Normenreste, die keine Normen sind, als Erzeuger nichtabelscher
K{\"o}rpererweiterungen.~II}, J.~Reine Angew.~Math.~\textbf{182} (1940), 217--234.
\bibitem[Ser77]{Serre1977}
J.-P.~Serre, \emph{Linear representations of finite groups}, Springer-Verlag, New York-Heidelberg, 1977, translated from the second French edition by L.~L.~Scott, Gard.~Texts in Math., Vol.~42. 
\bibitem[Ser79]{Serre1979}
J.-P.~Serre, \emph{Local Fields}, Grad.~Texts in Math.,Vol.~67, Springer, 1979. 
\bibitem[Suz82]{Suzuki1982}
M.~Suzuki, \emph{Group theory, I}, Grundlehren der Mathematischen Wissenschaften, 247, Springer-Verlag, Berlin-New York, 1982.
\bibitem[Tah72]{Tahara1972}
K.~Tahara, \emph{On the second cohomology groups of semidirect products}, Math.~Z.~\textbf{129} (1972), 365--379. 
\bibitem[Tat67]{Tate1967}
J.~Tate, \emph{Global class field theory}, Algebraic Number Theory, Proceedings of an instructional conference organized by the London Mathematical Society (a NATO Advanced Study Institute) with the support of the International Mathematical Union, Edited by J.~W.~S.~Cassels and F.~Fl{\"o}lich, 162--203, Academic Press, London; Thompson Book Co., Inc., Washington, D.C.~1967. 
\bibitem[Vos69]{Voskresenskii1969}
V.~E.~Voskresenskii, \emph{The birational equivalence of linear algebraic groups}, Dokl.~Akad.~Nauk SSSR \textbf{188} (1969), 978--981; erratum, ibid.~191 1969 nos., 1, 2, 3, vii; translation in Soviet Math.~Dokl.~\textbf{10} (1969), 1212--1215. 
\bibitem[VK84]{Voskresenskii1984}
V.~E.~Voskresenskii, B.~E,~Kunyavskii, \emph{Maximal tori in semisimple algebraic groups}, Kuibyshev State Inst., Kuibyshev (1984). Deposited in VINITI March 5 1984, No.~1269--84 Dep.~(Ref.~Zh.~Mat.~(1984), 7A405 Dep).  
\bibitem[Vos98]{Voskresenskii1998}
V.~E.~Voskresenskii, \emph{Algebraic groups and their birational invariants}, translated from the Russian manuscript by B.~Kunyavskii, Trans.~Math.~Monographs, 179, Amer.~Math.~Soc.~Providence, RI, 1998. 
\bibitem[Win72]{Winter1972}
S.~L.~Winter, \emph{The automorphism group of an extraspecial $p$-group}, Rocky Mountain J.~Math., \textbf{2} (1972), 159--168. 
\end{thebibliography}
\end{document}